\documentclass[reqno, 11pt,fleqn]{amsart} 
\usepackage{graphicx}
\usepackage{amsmath,amssymb,amsthm, mathrsfs, mathtools,bbm}
\usepackage[usenames,dvipsnames]{xcolor}

\usepackage{enumerate}
\usepackage{tabularx}
\usepackage{ltablex}
\usepackage{booktabs}

\usepackage{hyperref}
\hypersetup{
    colorlinks,
    linkcolor={red!60!black},
    citecolor={green!60!black},
    urlcolor={blue!60!black}
}
\usepackage[T1]{fontenc}
\usepackage{lmodern}
\usepackage[babel]{microtype}
\usepackage[english]{babel}
\usepackage{comment}

\usepackage{enumitem}
\usepackage{arydshln}

\usepackage{geometry}
\graphicspath{{images/}} 

\usepackage{theoremref}

\theoremstyle{plain}
\newtheorem{thm}{Theorem}[section]
\newtheorem{theorem}[thm]{Theorem}
\newtheorem*{theorem*}{Theorem}
\newtheorem{prop}[thm]{Proposition}

\newtheorem{clm}[thm]{Claim}
\newtheorem{cor}[thm]{Corollary}

\newtheorem{lemma}[thm]{Lemma}

\theoremstyle{definition}
\newtheorem{definition}[thm]{Definition}
\newtheorem{rem}[thm]{Remark}
\newtheorem{problem}[thm]{Problem}

\numberwithin{equation}{section}

\DeclarePairedDelimiter{\parens}{(}{)}
\DeclarePairedDelimiter{\set}{\{}{\}}
\DeclarePairedDelimiter{\brackets}{[}{]}
\DeclarePairedDelimiter{\floor}{\lfloor}{\rfloor}
\DeclarePairedDelimiter{\ceil}{\lceil}{\rceil}
\DeclarePairedDelimiter\abs{\lvert}{\rvert}   
\DeclarePairedDelimiter\size{\lvert}{\rvert}   

\renewcommand{\leq}{\leqslant}
\renewcommand{\geq}{\geqslant}
\renewcommand{\Pr}{\mathbb{P}}

\renewcommand{\phi}{\varphi}
\newcommand{\eps}{\varepsilon} 

\newcommand{\cM}{\ensuremath{\mathcal{M}}} 
\newcommand{\cS}{\ensuremath{\mathcal{S}}} 
\newcommand{\cT}{\ensuremath{\mathcal{T}}} 

\newcommand{\nto}{\nrightarrow}

\newcommand{\CycleColors}{\ensuremath{\cS(C_\ell)}}
\newcommand{\CliqueColors}{\ensuremath{\cS(K_t)}}
\newcommand{\AllColors}{\ensuremath{\cS}}

\newcommand{\Bollobas}{Bollob{\'a}s}
\newcommand{\Boyadzhiyska}{Boyadzhiyska}
\newcommand{\Burr}{Burr}
\newcommand{\Chvatal}{Chv{\'a}tal}
\newcommand{\Clemens}{Clemens}

\newcommand{\Erdos}{Erd\H{o}s}
\newcommand{\Grinshpun}{Grinshpun}
\newcommand{\Gupta}{Gupta}
\newcommand{\Han}{H{\`a}n}

\newcommand{\Lovasz}{Lov{\'a}sz}

\newcommand{\Nesetril}{Ne\v{s}et\v{r}il}

\newcommand{\Rodl}{R\"{o}dl}
\newcommand{\Rogers}{Rogers}

\newcommand{\Szabo}{Szab\'o}
\newcommand{\Szekeres}{Szekeres}
\newcommand{\Szemeredi}{Szemer{\'e}di}

\newcommand{\Zumstein}{Zumstein}
\newcommand{\Zurcher}{Z\"urcher}

\newcommand{\FGLPS}{Fox, \Grinshpun{}, Liebenau, Person, and \Szabo{}}

\usepackage{tikz}
\usepackage{calc}
\usetikzlibrary{calc}

\tikzstyle{vertex}=[circle, draw, fill=black, inner sep=0pt, minimum size=6pt]
\tikzstyle{smallvertex}=[circle, draw, fill=black, inner sep=0pt, minimum size=4pt] 
\tikzstyle{tinyvertex}=[circle, draw, fill=black, inner sep=0pt, minimum size=1pt] 
\newcommand{\vertex}{\node[vertex]}
\newcommand{\smallvertex}{\node[smallvertex]}

\newcounter{Angle}
\newcounter{BaseLoop}

\newcommand{\widthedge}{1.5} 

\newcommand{\TikzExampleKrCl}[2]{
    \begin{tikzpicture}[x=#1 cm, y=#2 cm]
        \vertex (a1) at (-15:1) [label=0:]{};
        \smallvertex (b1) at (-5:1.1) [label=0:]{};
        \smallvertex (c1) at (5:1.1)  [label=0:]{};        
        \vertex (d1) at (15:1)  [label=0:]{};
        
        \vertex (a2) at (75:1) [label=0:]{};
        \smallvertex (b2) at (85:1.1) [label=0:]{};
        \smallvertex (c2) at (95:1.1)  [label=0:]{};        
        \vertex (d2) at (105:1)  [label=0:]{};
        
        \vertex (a3) at (165:1) [label=0:]{};
        \smallvertex (b3) at (175:1.1) [label=0:]{};
        \smallvertex (c3) at (185:1.1)  [label=0:]{};        
        \vertex (d3) at (195:1)  [label=0:]{};
        
        \vertex (a4) at (-105:1) [label=0:]{};
        \smallvertex (b4) at (-95:1.1) [label=0:]{};
        \smallvertex (c4) at (-85:1.1)  [label=0:]{};        
        \vertex (d4) at (-75:1)  [label=0:]{};

        \foreach \i in {1, 2, 3}{
            \setcounter{BaseLoop}{\i+1}
            \foreach \j in {\theBaseLoop,..., 4}{
            \path
            (a\i) edge[color=red, line width=\widthedge pt] (a\j) 
            (a\i) edge[color=red, line width=\widthedge pt] (d\j) 
            (d\i) edge[color=red, line width=\widthedge pt] (a\j)
            (d\i) edge[color=red, line width=\widthedge pt] (d\j); }}
        
        \foreach \i in {1, 2, 3, 4}{
        \path
        (a\i) edge[color=blue, dashed,line width=\widthedge pt] (b\i)
        (b\i) edge[color=blue, dashed,line width=\widthedge pt] (c\i)
        (c\i) edge[color=blue, dashed,line width=\widthedge pt] (d\i);}

        \draw[draw=black] (0.75,-0.6) rectangle ++(1.9,-0.4);
        \node (red1) at (0.8,-0.7) [label=0:]{};
        \node (red2) at (1.2,-0.7) [label=0:signal edges - red-determiners]{};
        \path  (red1) edge[color=red, line width=\widthedge pt](red2) ;
        \node (blue1) at (0.8,-0.9) [label=0:]{};
        \node (blue2) at (1.2,-0.9) [label=0:signal edges - blue-determiners]{};
        \path  (blue1) edge[color=blue, dashed, line width=\widthedge pt] (blue2) ;
        
        \vertex (v) at (45:1.2) [label=45:$v$,color=gray]{}; 
        \foreach \i in {1, 2, 3, 4}{
        \path 
        (v) edge[color=gray] (a\i)
        (v) edge[color=gray] (d\i); }
        
    \end{tikzpicture}

}

\newcommand{\TikzExampleCkCl}[2]{
\begin{tikzpicture}[x=#1 cm, y=#2 cm]
	\foreach \i in {1, 2, 3, 4, 5, 6} {
		\setcounter{Angle}{360 - \i * 360 / 6 + 90}
		\vertex (c\i) at (\theAngle:0.666) [label=\theAngle:]{};
	}
	\vertex (x) at (90:0.666) [label=90:$x$]{};	
	\vertex (y) at (210:0.666) [label=210:$y$]{};	
	\vertex (z) at (330:0.666) [label=330:$z$]{};	

	\path 
	(c1) edge[color=red, line width=\widthedge pt] (c2)
	(c2) edge[color=red, line width=\widthedge pt] (c3)
	(c3) edge[color=red, line width=\widthedge pt] (c4)
	(c4) edge[color=red, line width=\widthedge pt] (c5)
	(c5) edge[color=red, line width=\widthedge pt] (c6)
	(c6) edge[color=red, line width=\widthedge pt] (c1);
	
	\vertex (b) at (15:1) [label=20:]{};
	\vertex (a) at (45:1) [label=40:]{};
	\path 
	(c6) edge[color=blue, dashed, line width=\widthedge pt] (a)
	(a) edge[color=blue, dashed, line width=\widthedge pt](b)
	(b) edge[color=blue, dashed, line width=\widthedge pt] (c2);
	
	\vertex (c) at (-75:1) [label=20:]{};
	\vertex (d) at (-105:1) [label=40:]{};
	\path 
	(c2) edge[color=blue, dashed, line width=\widthedge pt] (c)
	(c) edge[color=blue, dashed, line width=\widthedge pt] (d)
	(d) edge[color=blue, dashed, line width=\widthedge pt] (c4);
	
	\vertex (e) at (165:1) [label=20:]{};
	\vertex (f) at (135:1) [label=40:]{};
	\path 
	(c4) edge[color=blue, dashed, line width=\widthedge pt] (e)
	(e) edge[color=blue, dashed, line width=\widthedge pt](f)
	(f) edge[color=blue, dashed, line width=\widthedge pt] (c6);
	
    \draw[draw=black] (0.6,-0.6) rectangle ++(1.9,-0.4);
    \node (red1) at (0.65,-0.7) [label=0:]{};
    \node (red2) at (1.05,-0.7) [label=0:signal edges - red-determiners]{};
    \path  (red1) edge[color=red, line width=\widthedge pt](red2) ;
    \node (blue1) at (0.65,-0.9) [label=0:]{};
    \node (blue2) at (1.05,-0.9)  [label=0:signal edges - blue-determiners]{};
    \path  (blue1) edge[color=blue, dashed, line width=\widthedge pt] (blue2) ;
    
   \vertex (v) at (0,0) [label=45:$v$,color=gray]{}; 
	\path 
	(v) edge[color=gray] (x)
	(v) edge[color=gray] (y)
	(v) edge[color=gray] (z);    
	
\end{tikzpicture}}

\title[Minimum degree of asymmetric Ramsey-minimal graphs]
{On the minimum degree of minimal Ramsey graphs for cliques versus cycles}

\author[A.Bishnoi]{Anurag Bishnoi}
\address[A1]{Delft University of Technology, Delft, Netherlands}
\email{a.bishnoi@tudelft.nl}

\author[S.Boyadzhiyska]{Simona Boyadzhiyska}
\address[A2]{Freie Universit\"at Berlin, Institut f\"ur Mathematik, Berlin, Germany}
\email{s.boyadzhiyska@fu-berlin.de}

\author[D.Clemens \and P.Gupta]{Dennis Clemens \and Pranshu Gupta}
\address[A3,A4]{Hamburg University of Technology, Institute of Mathematics, Hamburg, Germany}
\email{\{dennis.clemens,pranshu.gupta\}@tuhh.de}

\author[T.Lesgourgues \and A.Liebenau]{Thomas Lesgourgues \and Anita Liebenau} 
\address[A4,A5]{School of Mathematics and Statistics, UNSW Sydney, 2052 Kensington, NSW, Australia}
\email{\{t.lesgourgues,a.liebenau\}@unsw.edu.au}

\begin{document}
\maketitle 
\setlength{\mathindent}{2cm}


\begin{abstract}
A graph $G$ is said to be \emph{$q$-Ramsey} for a $q$-tuple of graphs $(H_1,\ldots,H_q)$, denoted by $G\to_q(H_1,\ldots,H_q)$, if every $q$-edge-coloring of $G$ contains a monochromatic copy of $H_i$ in color $i,$ for some $i\in[q]$. 
Let $s_q(H_1,\ldots,H_q)$ denote the smallest minimum degree of $G$ over all graphs $G$ that are minimal $q$-Ramsey for $(H_1,\ldots,H_q)$ (with respect to subgraph inclusion). 
The study of this parameter was initiated in 1976 by \Burr{}, \Erdos{} and \Lovasz{},  who determined its value precisely for a pair of cliques.
Over the past two decades the parameter $s_q$ has been studied by several groups of authors, the main focus being  on the symmetric case, where $H_i\cong H$ for all $i\in [q]$. The asymmetric case, in contrast, has received much less attention. In this paper, we make progress in this direction, studying asymmetric tuples consisting of cliques, cycles and trees. We determine $s_2(H_1,H_2)$ when $(H_1,H_2)$ is a pair of one clique and one tree, a pair of one clique and one cycle, and when it is a pair of two different cycles. We also generalize our results to multiple colors and obtain bounds on $s_q(C_\ell,\ldots,C_\ell,K_t,\ldots,K_t)$ in terms of the size of the cliques $t$, the number of cycles, and the number of cliques. Our bounds are tight up to logarithmic factors when two of the three parameters are fixed. 
\end{abstract}

\section{Introduction}\label{sec:introduction}

A graph $G$ is said to be $q$-\emph{Ramsey} for a $q$-tuple of graphs $(H_1,\ldots,H_q)$, denoted by \mbox{$G\to_q (H_1,\ldots,H_q)$,} if, for every $q$-coloring of the edges of $G$, there exists a monochromatic copy of $H_i$ in color $i$ for some $i\in[q]$. In the symmetric case, when $H_i\cong H$ for all $i\in [q]$, we simply say that the graph $G$ is $q$-\emph{Ramsey} for $H$. It follows from Ramsey's theorem~\cite{RamseyFormalLogic} that such a graph $G$ exists for any choice of $(H_1,\ldots,H_q)$. The most well-known object of study in this area is arguably the \emph{Ramsey number} of a $q$-tuple of graphs $(H_1,\ldots,H_q)$, denoted by $r_q(H_1,\ldots,H_q)$ and defined as the smallest number of vertices in any graph that is $q$-Ramsey for $(H_1,\ldots,H_q)$. Despite being studied intensively for many families of graphs, it has been determined for very few of them. The case where each $H_i$ is isomorphic to a complete graph $K_t$ is of particular interest. Early results by \Erdos{}~\cite{erdos1947} and \Erdos{} and \Szekeres{}~\cite{erdos_combinatorial_1935} establish that $2^{t/2}\leq r_2(K_t,K_t)\leq 4^t$. Despite being over seventy years old, these bounds have only been improved by subexponential factors: the best known lower bound is due to Spencer~\cite{spencer1975ramsey}, while the best known upper bound was established very recently by Sah~\cite{sah2020diagonal}, improving on a previous result due to Conlon~\cite{conlon_upper_bound}.

A natural generalization is to investigate other graph parameters. In their seminal paper~\cite{burr_graphs_1976}, \Burr{}, \Erdos{}, and \Lovasz{} initiated the study of minimum degrees of Ramsey graphs. Observe that, given any graph $G$ that is $q$-Ramsey for $H$, we can add an isolated vertex to $G$ to obtain another graph $G'$ that is also $q$-Ramsey for $H$, with minimum degree zero. To avoid such trivialities, we restrict our attention to graphs $G$ that are minimal in the following sense. A graph $G$ is said to be $q$-\emph{Ramsey-minimal} for $(H_1,\ldots,H_q)$ if $G$ is $q$-Ramsey for $(H_1,\ldots,H_q)$ but no proper subgraph of $G$ is. We denote the family of all $q$-Ramsey-minimal graphs for $(H_1,\ldots,H_q)$ by $\cM_q(H_1,\ldots,H_q)$. We are interested in studying the parameter $s_q(H_1,\ldots,H_q)$, defined as the smallest minimum degree among all $q$-Ramsey-minimal graphs for $(H_1,\ldots,H_q)$, that is, $s_q(H_1,\ldots,H_q) = \min \set*{\delta(G) : G\in\cM_q(H_1,\ldots,H_q)}$, where $\delta(G)$ denotes the minimum degree of $G$. In the symmetric case, when $H_i\cong H$ for all $i\in [q]$, we simply write $s_q(H)$ instead of $s_q(H,\ldots,H)$ (and similarly for $r_q(H)$ and $\cM_q(H)$). It is not difficult to show that
\begin{equation}
    \sum\limits_{i=1}^{q}(\delta(H_i)-1) < s_q(H_1,\ldots,H_q)\leq r_q(H_1,\ldots,H_q)-1. \label{eq:trivial_bounds_sq}
\end{equation}
The proof for the symmetric case and when $q=2$ can be found in Fox and Lin~\cite[Theorem 3]{fox_minimum_2007}, and the argument easily extends to the more general inequalities. 

\Burr{}, \Erdos{}, and \Lovasz{}~\cite{burr_graphs_1976} considered pairs of complete graphs and established that  $s_2(K_t,K_k)=(t-1)(k-1)$. We want to remark that, in the symmetric case, there is a large gap between $s_2(K_t)$ and the exponential upper bound in~\eqref{eq:trivial_bounds_sq}. This surprising phenomenon tells us that, while every graph that is $2$-Ramsey for $K_t$ must have at least exponentially many vertices, there is such a graph $G$ that contains a vertex of degree quadratic in $t,$ and this vertex is essential for the Ramsey property of $G.$  

Since the  seminal article of \Burr{}, \Erdos{}, and \Lovasz{}~\cite{burr_graphs_1976}, the parameter $s_2(H)$ has been studied  for various graphs $H$. For example, Fox and Lin~\cite{fox_minimum_2007} showed that the lower bound in~\eqref{eq:trivial_bounds_sq} is tight for complete bipartite graphs. \Szabo{}, \Zumstein{}, and \Zurcher{}~\cite{szabo_minimum_2010} extended this result to several other classes of bipartite graphs, including trees and even cycles, while \Grinshpun{}~\cite{grinshpun_problems_2015} proved it for $3$-connected bipartite graphs. Some non-bipartite cases were addressed as well, such as cliques with pendant edges~\cite{fox_what_2014}, cliques with the edge set of a star removed~\cite{grinshpun_minimum_2017}, and odd cycles~\cite{boyadzhiyska2020minimal}.

All these results address the symmetric case and, to the best of our knowledge, the result of \Burr{}, \Erdos{}, and \Lovasz{} concerning pairs of cliques is the only asymmetric case to date. It is then natural to consider pairs of graphs $(K_t,H),$ where $H$ is a very sparse graph such as a tree $T_{\ell}$ or a cycle $C_{\ell}$ (where $\ell$ is the number of vertices). These pairs have already been studied in Ramsey theory, in the context of Ramsey numbers. A classical result by \Chvatal{}~\cite{Chvtal1977TreecompleteGR} states that $r_2(K_t,T_\ell)=(t-1)(\ell-1)+1.$ In fact, any red/blue-coloring witnessing the inequality $r_2(K_t,T_\ell)>(t-1)(\ell-1)$ is so special that we can easily deduce the following. 
\begin{prop}\thlabel{thm:2_colors_completetree}
For all integers $t\geq3$ and $\ell\geq 2$, we have $s_2(K_t,T_\ell)=t-1.$
\end{prop}
The Ramsey number $r_2(K_t,C_{\ell})$ has received considerably more attention, as it shows different behaviour depending on the magnitude of~$\ell$; after decades of effort by researchers, the study of these Ramsey numbers has culminated in several very recent breakthroughs.  The case when $\ell=3$ defaults to the notoriously difficult case of the asymmetric Ramsey number $r_2(K_t,K_3)$ which is equal to $(4+o(1)) t^2/\log t$, as shown by Bohman and Keevash~\cite{bk2021},  Fiz Pontiveros, Griffiths, and Morris~\cite{fpgm2013}, and Shearer~\cite{shearer1983}, following the earlier results by Ajtai, Koml\'os, and Szemer\'edi~\cite{aks1980} and by Kim~\cite{kim1995}. At the other end of the spectrum, Keevash, Long, and Skokan~\cite{kls2021} showed that $r_2(K_t,C_{\ell})=(t-1)(\ell-1)+1$ for $\ell=\Omega(\log t/\log \log t),$ and that this bound on $\ell$ is best possible for the equality to hold. For a more detailed discussion on the history of $r_2(K_t,C_{\ell})$ we refer the reader to~\cite{kls2021}. We determine the value of $s_2(K_t, C_{\ell})$ precisely, showing that, unlike the Ramsey number, our parameter of interest is independent of $\ell$.  

We also complete the study $s_2$ for pairs of graphs each of which is a complete graph or a cycle by determining $s_2(C_k, C_\ell)$. The study of the Ramsey number in this case was completed already in the 1970s by Rosta~\cite{rosta1973ramsey} and Faudree and Schelp~\cite{faudree1974all}, and again depends on the values of $k$ and $\ell$. The minimum degree $s_2$, however, is again independent of either cycle length.

\begin{theorem}\thlabel{thm:2_colors_thm}
For all integers $t\geq 3$ and $k, \ell \geq 4$, 
\begin{enumerate}[label=(\roman*), ref=\ref{thm:2_colors_thm}~(\roman*)]
    \item \(s_2(C_k,C_{\ell})=3\).\thlabel{thm:2_colors_cyclecycle}
    \item \(s_2(K_t,C_{\ell})=2(t-1)\). \thlabel{thm:2_colors_completecycle}
\end{enumerate}
\end{theorem}

Next, we venture into the multicolor setting. 
\Boyadzhiyska{}, \Clemens{}, and \Gupta{}~\cite{boyadzhiyska2020minimal} showed that 
$s_q(C_\ell)= q+1$ for all $q\geq 2$ and $\ell\geq 4$. 
The only other case that has been studied deals with symmetric tuples of cliques, and no precise values are known for $s_q(K_t)$ for $q>2.$
\FGLPS{}~\cite{fox2016minimum} showed that $s_q(K_t)$ is quadratic in $q$, up to a polylogarithmic factor, when the size of the clique is fixed. The polylogarithmic factor was settled to be $\Theta(\log q)$ when $t=3$ by Guo and Warnke~\cite{Guo-Warnke20}, following earlier work in~\cite{fox2016minimum}. In the other regime, when the number of colors is fixed, \Han{}, \Rodl{}, and  \Szabo{}~\cite{Han:2018aa} showed that $s_q(K_t)$ is quadratic in the clique size $t$, up to a polylogarithmic factor. Bounds that are polynomial in both $q$ and $t$ are also known, see~\cite{fox2016minimum} and Bamberg, Bishnoi, and Lesgourgues~\cite{bamberg2020minimum}. 

In this paper, we investigate the parameter $s_q$ in the case of multiple cliques and multiple cycles. For given integers $q,q_1,q_2\geq 0$ with $q = q_1+q_2$, $t\geq 3$, and $\ell\geq 4$, we define $\cT = \cT(q_1,q_2,\ell,t)$ to be the $q$-tuple consisting of $q_1$ cycles on $\ell$ vertices and $q_2$ cliques on $t$ vertices, that is,
\begin{equation}\label{eq:definition_cT}
\cT(q_1,q_2,\ell,t)=(\underbrace{C_\ell,\dots, C_\ell}_{q_1 \text{ times}}, \underbrace{K_t,\dots, K_t}_{q_2 \text{ times}}),    
\end{equation}
and let $s_q(\cT(q_1,q_2,\ell,t))$ be the smallest minimum degree of a $q$-Ramsey-minimal graph for $\cT(q_1,q_2,\ell,t)$. When the parameters are clear from context, we will suppress them from the notation. Our main result in the multicolor setting is the following.

\begin{theorem}\thlabel{thm:sq_relation}
For all $\ell\geq 4$, $t\geq 3$, and all $q,q_1,q_2\geq 1$ such that $q_1+q_2 = q$, we have
    \begin{equation}\label{eq:sq_relation}
        s_{q_2}(K_t) + q_1 \leq s_q(\cT(q_1,q_2,\ell,t)) \leq s_q(K_t).
    \end{equation}
\end{theorem}

Note that these upper and lower bounds are independent from the cycles' length $\ell.$ In fact, we prove a stronger statement in  \thref{lem:packing_equivalence} from which it follows that  $s_q(\cT)$ itself does not depend on~$\ell$. Using the known bounds for $s_{q}(K_t)$, we can deduce the following corollary. 

\begin{cor}\thlabel{cor:bounds}\hfill
\begin{enumerate}[label=(\roman*)]
    \item For all $t\geq 4$ and $q_1\geq 1$, there exist constants $c,C>0$ such that, for all $\ell\geq 4$ and $q_2\geq 1$, we have
    \[ c\, q_2^2\frac{\log q_2}{\log\log q_2} \leq s_{q_1+q_2}(\cT(q_1,q_2,\ell,t)) \leq Cq_2^2 (\log q_2)^{8(t-1)^2}.\] \label{cor:bounds:q2large}
    
    \item For all $q_1\geq 1$ there exist constants $c,C>0$ such that, for all $\ell\geq 4$ and $q_2\geq 1$, we have
    \[ c\,q_2^2\log q_2 \leq s_{q_1+q_2}(\cT(q_1,q_2,\ell, 3))\leq Cq_2^2\log q_2.\]\label{cor:bounds:triangle}
    
    \item For all $q_1,q_2\geq 1$, there exists a constant $C>0$ such that, for all $\ell\geq 4$ and $t\geq 3$, we have
   \[(t-1)^2 \leq s_{q_1+q_2}(\cT(q_1,q_2,\ell,t)) \leq Ct^2 \log^2 t.\]\label{cor:bounds:tlarge}
\end{enumerate}
\end{cor}

Thus, \thref{thm:sq_relation} is sufficient to determine $s_q(\cT(q_1,q_2,\ell,t))$ in terms of $q_2$ and in terms of $t$ when the other parameters are fixed. Similarly, the bounds in~\cite{bamberg2020minimum,fox2016minimum} yield bounds on $s_{q_1+q_2}(\cT(q_1,q_2,\ell,t))$ that are polynomial in both $t$ and $q$.

When $q_1$ is large compared to the other parameters, then 
the lower bound of \eqref{eq:sq_relation} is linear in $q_1$ while the upper bound is essentially quadratic in $q_1$. In this case, using the already mentioned stronger statement of~\thref{lem:packing_equivalence}, we prove the following asymptotically optimal result.

\begin{theorem}\thlabel{thm:q1large}
For all $\ell\geq 4$, $t\geq 3$, $q_2\geq 1$, and $\varepsilon>0$, there exists $q_0$ such that for all $q_1\geq q_0,$ we have 
\[s_{q_1+q_2}(\cT(q_1,q_2,\ell, t)) \leq (1+\varepsilon)q_1.\]
\end{theorem}

\paragraph{\textbf{Organization of the paper.}}
In Section~\ref{sec:preliminaries}, we introduce some of the key definitions and known results that will be necessary in the rest of the paper, and state our main technical results, \thref{lem:cycle_and_clique_determiners,lem:cycle_and_clique_senders}. Section~\ref{sec:2colors} is dedicated to the proofs of the $2$-color cases (\thref{thm:2_colors_completetree,thm:2_colors_thm}). In Section~\ref{sec:proof_main_result} we prove~\thref{thm:sq_relation,thm:q1large}, assuming the existence of certain gadget graphs as guaranteed by~\thref{lem:cycle_and_clique_determiners,lem:cycle_and_clique_senders}. Finally, Section~\ref{sec:existence_set_signal} contains the proof of~\thref{lem:cycle_and_clique_determiners,lem:cycle_and_clique_senders}.

\section{Preliminaries}\label{sec:preliminaries}

In this section, we introduce notation and key ideas that will be used throughout the article, and state our main technical results, the existence of gadget graphs for a $q$-tuple of cycles and cliques (\thref{lem:cycle_and_clique_determiners,lem:cycle_and_clique_senders}).\smallskip

We use standard graph theoretic notation throughout the article. Given a hypergraph $G$, we write $v(G)$ for the size of its vertex set and $e(G)$ for the size of its edge set. We often identify a graph with its edge set. In particular, for two graphs $G$ and $H$, we use $G-H$ to denote the graph on $V(G)$ with edge set $E(G)\setminus E(H)$. We say that a graph is \emph{$H$-free} if it does not contain $H$ as a (not necessarily induced) subgraph. The \emph{distance} between two sets of vertices $A$ and $B$ in a graph is the length of a shortest path with one endpoint in $A$ and one endpoint in $B$. 

Unless otherwise specified, we use the term coloring to refer to an edge-coloring. If a coloring of a graph uses at most $q$ colors, then we say that it is a $q$-coloring; unless otherwise specified, the color palette in a $q$-coloring is taken to be the set $[q]=\set*{1,\ldots,q}$. When $q=2$, we call the first color red and the second color blue. Given a $q$-coloring $\phi$ of a graph $G$ and a subgraph $F\subseteq G$, we will write $\phi_{|F}$ for the $q$-coloring induced by $\phi$ on the edges of $F$. 
Given a $q$-tuple of graphs $(H_1,\ldots, H_q)$, we say that a $q$-coloring $\phi$ of a graph $G$ is \emph{$(H_1,\ldots, H_q)$-free} if, for all $i\in [q]$, the graph $\phi^{-1}(\set{i})$ is $H_i$-free. When $H_i\cong H$ for all $i\in [q]$, we will simply say that $\phi$ is \emph{$H$-free} when $\phi$ is $(H,\dots, H)$-free.

Given colorings $\phi_1$ and $\phi_2$ of $G_1$ and $G_2$, respectively, such that $\phi_1(e) = \phi_2(e)$ for all $e\in E(G_1)\cap E(G_2)$, we define the coloring $\phi_1\cup\phi_2$ on $G_1\cup G_2$ by setting 
\[\phi(e) = 
\begin{cases}
    \phi_1(e) &\text{ if } e\in E(G_1),\\
    \phi_2(e) &\text{ if } e\in E(G_2).
\end{cases}\]

Let  $t\geq 3$, $\ell\geq 4$ and $q, q_1, q_2\geq 0$ be integers such that $q=q_1+q_2$. Recall that $\cT=\cT(q_1,q_2,\ell, t)$ denotes the $q$-tuple of cycles and cliques as defined in \eqref{eq:definition_cT}. For convenience, we will sometimes write \CycleColors{} for the color palette $\{1,\ldots,q_1\}$ and refer to it as the \emph{cycle-colors}; similarly, \CliqueColors{} will denote the color palette $\{q_1+1,\ldots,q\}$, referred to as the \emph{clique-colors}.

\subsection{Signal senders and determiners}\label{subsec:definition_gadgets}

For our constructions, we need gadget graphs similar to those introduced by \Burr{}, \Erdos, and \Lovasz~\cite{burr_graphs_1976} and \Burr{}, Faudree, and Schelp~\cite{burr1977ramseyminimal}. Let $q\geq 2$ and $(H_1,\dots, H_q)$ be a $q$-tuple of graphs. We begin with the simpler of the two gadget graphs.

\begin{definition}[Set-determiner]\thlabel{def:safe_set_determiner}
Let $X\subseteq [q]$ be any subset of colors. An \emph{$X$-determiner} for $(H_1,\dots, H_q)$  is a graph $D$ with a distinguished edge $d$ satisfying the following properties:  
\begin{enumerate}[label=(D\arabic*)]
    \item $D\nrightarrow_q (H_1,\dots, H_q)$.\label{axiom:set_determiner_not_Ramsey}
    \item For any $(H_1,\dots, H_q)$-free coloring $\phi$ of $D$, we have $\phi(d)\in X$.\label{axiom:set_determiner_colored_edge}
    \item For any color $c\in X$, there exist an $(H_1,\dots, H_q)$-free coloring $\phi$ of $D$ such that $\phi(d)=c$.\label{axiom:set_determiner_any_color}
\end{enumerate}
The edge $d$ is referred to as the \emph{signal edge} of $D$. 
\end{definition}

In the special case where $X=\set{c}$, these gadgets were defined by \Burr{}, Faudree, and Schelp in~\cite{burr1977ramseyminimal}, and are simply called \emph{determiners}. It is not difficult to see that a $\set{c}$-determiner can only exist for a $q$-tuple $(H_1,\dots, H_q)$ if $H_c \not\cong H_i$ for all $i\in[q]\setminus\set{c}$. Determiners are known to exist for all pairs $(G,H)$ such that $G\not\cong H$ and $G$ and $H$ are $3$-connected (see \Burr{}, \Nesetril{}, and \Rodl{}~\cite{burr1985useofsenders}). More recently, they were shown to exist for pairs of the form $(C_\ell, H)$ by Siggers~\cite{siggers_non-bipartite_2014}, where $H$ is a $2$-connected graph satisfying some additional properties.\smallskip

While set-determiners allow us to pick which set the color of a certain edge should come from, in order to have control over the specific color pattern we see on a group of edges (e.g., which edges should have the same color), we also define the following more sophisticated gadgets.

\begin{definition}[Set-sender]\thlabel{def:safe_set_sender}
Let $X\subseteq [q]$ be any subset of colors. A \emph{negative} (respectively \emph{positive}) \emph{$X$-sender} for $(H_1,\dots, H_q)$ is a graph $S$ with distinguished edges $e$ and $f$, satisfying the following properties:
\begin{enumerate}[label=(S\arabic*)]
    \item $S\nrightarrow_q (H_1,\dots, H_q)$.\label{axiom:set_sender_not_Ramsey}
    \item For any $(H_1,\dots, H_q)$-free coloring $\phi$ of $S$, there exist colors $c_1,c_2\in X$ with $c_1\neq c_2$ (respectively $c_1= c_2$) such that $\phi(e)=c_1$ and $\phi(f)=c_2$. \label{axiom:set_sender_colored_edges} 
    \item For any colors $c_1,c_2\in X$ with $c_1\neq c_2$ (respectively $c_1= c_2$), there exists an $(H_1,\dots, H_q)$-free coloring $\phi$ of $S$ with $\phi(e)=c_1$ and $\phi(f)=c_2$.\label{axiom:set_sender_any_color}
\end{enumerate}
The edges $e$ and $f$ are referred to as the \emph{signal edges} of $S$.
\end{definition}

In the special case where $X=[q]$, these gadgets were introduced by \Burr{}, \Erdos{}, and \Lovasz{}~\cite{burr_graphs_1976} and are called \emph{signal senders}. In~\cite{burr_graphs_1976} and~\cite{burr1977ramseyminimal}, it was shown that positive and negative signal senders exist for pairs of complete graphs. Subsequently, it was proved that they exist for other graphs as well as for more colors; in particular, \Rodl{} and Siggers~\cite{rodl_ramsey_2008} and Siggers~\cite{siggers_highly_2008} established their existence for any number of colors when $H_i\cong H$ for all $i\in[q]$  and $H$ is either $3$-connected or a cycle. In a later paper, Siggers~\cite{siggers_non-bipartite_2014} showed the existence of signal senders for some pairs of the form $(C_\ell, H)$.

In the symmetric case, when $H_i\cong H$ for all $i\in[q]$, we write set-senders for $H$ to denote set-senders for $(H,\dots, H)$, and similarly for signal senders. Additionally when $q=2$, we simplify the notation and write \emph{red-determiners} (respectively \emph{blue-determiners}) for \{red\}-determiners (respectively \{blue\}-determiners).\medskip

Intuitively speaking, the utility of set-senders and set-determiners comes from the fact that these gadgets allow us to force specific color patterns on particular sets of edges. In our constructions, we usually start with a base graph $G$ and add set-senders and set-determiners so that, in any $(H_1,\dots,H_q)$-free coloring of the resulting graph, we obtain a particular color pattern on the edges of $G$. More precisely, we will say that we \emph{attach} a set-determiner $D$ to an edge $e$ of $G$ to mean that we create a new copy $\widehat{D}$ of $D$ such that $e$ is the signal edge of $\widehat{D}$, and $\widehat{D}$ is otherwise vertex-disjoint from $G$. Similarly, we will say that we \emph{connect} or \emph{join} two edges $e_1$ and $e_2$ of $G$ by a set-sender $S$ to mean that we create a new copy $\widehat{S}$ of $S$ such that $e_1$ and $e_2$ are the signal edges of $\widehat{S}$ (in an arbitrary fashion), and $\widehat{S}$ is otherwise vertex-disjoint from $G$.

In order for these constructions to be useful, we need to be able to control the new copies of $H_1,\dots, H_q$ that might be created in the process. In particular, since we usually use set-senders and set-determiners as black boxes, we would like to be able to obtain an $(H_1,\dots, H_q)$-free coloring of the entire graph by simply giving each of the building blocks an $(H_1,\dots, H_q)$-free coloring. This motivates the definition of a safe coloring, given by Siggers in~\cite{siggers_non-bipartite_2014}. 

\begin{definition}[Safe coloring]\thlabel{def:safe_coloring_2colors}
Let $F$ be a graph, $A\subseteq E(F)$, and $\phi$ be an $(H_1,\dots, H_q)$-free $q$-coloring of $F$.
We say that \emph{$\phi$ is safe at $A$} if, for any graph $G$ with $V(F)\cap V(G) \subseteq V(A)$, a $q$-coloring $\psi$ of $F\cup G$ with $\psi_{\mid F} = \phi$ is $(H_1,\dots, H_q)$-free if and only if $\psi_{\mid G}$ is $(H_1,\dots, H_q)$-free.
\end{definition}

We will call a set-sender (respectively set-determiner) \emph{safe} if the coloring guaranteed by property \ref{axiom:set_sender_any_color} (respectively \ref{axiom:set_determiner_any_color}) can be chosen to be safe at the signal edge(s).
\medskip

As explained above, in the asymmetric setting, the work of~\cite{burr_graphs_1976},~\cite{burr1977ramseyminimal}, and~\cite{burr1985useofsenders} established the existence of signal senders and determiners for pairs of the form $(H_1,H_2)$, where $H_1$ and $H_2$ are either $3$-connected or isomorphic to $K_3$. These determiners can be shown to be safe following an argument similar to \thref{rem:safeness}. The only other result in this direction that we are aware of is due to Siggers~\cite{siggers_non-bipartite_2014}, who used the ideas of \Bollobas{}, Donadelli, Kohayakawa, and Schelp~\cite{bollobas2001ramsey} to prove the existence of safe signal senders and safe determiners for many pairs of the form $(H, C_\ell)$, where $H$ is a $2$-connected graph satisfying certain technical properties. The special cases that are relevant to our $2$-color study in Section \ref{sec:2colors} are given in the following lemma. While~\thref{lem:siggers_safe_det_clique_cycle} also follows from our more general~\thref{lem:cycle_and_clique_determiners},~we briefly sketch Siggers' proof for both cases below, combining a few arguments from his paper.

\begin{lemma}[\cite{siggers_non-bipartite_2014}]
\hfill \thlabel{lem:siggers_safe_det}
\begin{enumerate}[label=(\roman*), ref=\ref{lem:siggers_safe_det}~(\roman*)]
    \item Let $k,\ell\geq 4$ be integers with $k<\ell$. Then there exist safe red-determiners and safe blue-determiners for $(C_k,C_\ell)$.\thlabel{lem:siggers_safe_det_cycles}
    \item Let $\ell \geq 4$ and $t\geq 3$. Then there exist safe red-determiners and safe blue-determiners for $(K_t,C_\ell)$.\thlabel{lem:siggers_safe_det_clique_cycle}
\end{enumerate}
\end{lemma}

\begin{proof}
We know that $C_k$ and $K_t$ are $2$-connected and, since $\ell > k$, these graphs contain no induced cycle of length at least $\ell+1$. Therefore, by \cite[Corollary~3.12]{siggers_non-bipartite_2014}, there exist safe red-determiners for $(C_k,C_\ell)$ and $(K_t, C_\ell)$.

Let $C$ be a copy of $C_k$, and let $e$ be any edge of $C$. Attach a copy of the safe red-determiner for $(C_k,C_\ell)$ from the previous paragraph to each edge of $C$ except $e$, and let $D$ be the resulting graph. Clearly, in any $(C_k,C_\ell)$-free coloring of $D$, the edge $e$ is blue. Furthermore, giving each copy of the red-determiner a safe $(C_k,C_\ell)$-free coloring, as guaranteed by property~\ref{axiom:set_determiner_any_color} and the safeness of that determiner, results in a $(C_k,C_\ell)$-free coloring of $D$. The safeness of the red-determiner further ensures that this coloring is safe at the edge $e$. Therefore $D$ is a safe blue-determiner for $(C_k,C_\ell)$, with signal edge $e$. A similar argument yields a safe blue-determiner for $(K_t,C_\ell)$.
\end{proof}\medskip

As explained in the introduction, in this paper we investigate the parameter $s_q$ in the case of multiple cliques and multiple cycles. Our main technical result stated below prove the existence of some set-determiners for such tuples of graphs. In its proof, we need the following results concerning the existence of signal senders in the symmetric setting, due to Siggers~\cite{siggers_highly_2008} and \Rodl{} and Siggers~\cite{rodl_ramsey_2008}, respectively.

\begin{lemma}[{\cite[Lemma 2.2]{siggers_highly_2008}}]\thlabel{lem:signal_senders_cycles}
For any $\ell\geq 4$ and any number of colors $q\geq 2$, there exist positive and negative signal senders for the cycle $C_\ell$ that have girth $\ell$ and distance at least $\ell+1$ between their signal edges.
\end{lemma}

\begin{lemma}[{\cite[lemma 2.2]{rodl_ramsey_2008}}]\thlabel{lem:signal_senders_cliques}
For any graph $H$ that is either $3$-connected or isomorphic to $K_3$, any number of colors $q\geq 2$, and any integer $d\geq 1$, there exist positive and negative signal senders for $H$ in which the signal edges are at distance at least $d$.
\end{lemma}

\begin{rem}\thlabel{rem:safeness}
We claim that the signal senders given by~\thref{lem:signal_senders_cycles,lem:signal_senders_cliques} are safe. First, let $S$ be a signal sender for $C_\ell$ with signal edges $e$ and $f$ and $F$ be any graph such that $V(F)\cap V(S) \subseteq e\cup f$. Let $\phi$ be a $C_\ell$-free coloring of $S$ and $\psi$ be a coloring of $S\cup F$ extending $\phi$. Suppose that $\psi_{|F}$ is also $C_\ell$-free but $\psi$ itself is not. This means that there exists a monochromatic copy $C$ of $C_\ell$ containing a vertex $v\in V(S)\setminus V(F)$ and a vertex $w\in V(F)\setminus V(S)$. There are two disjoint paths between $v$ and $w$ in $C$. Hence $C$ must contain two vertices from $V(e)\cup V(f)$. But $C$ cannot contain vertices from two different signal edges, since the distance between $e$ and $f$ in $S$ is at least $\ell+1$, so $C$ must contain both vertices of one signal edge, say $e$. But then $C\cup e$ contains a cycle of length less than $\ell$ that is fully contained in $S$, contradicting the fact that $S$ has girth $\ell$. Hence $\psi$ must be a $C_\ell$-free coloring. A similar argument shows that if $H$ is $3$-connected or isomorphic to $K_3$ and $S$ is a signal sender as given by~\thref{lem:signal_senders_cliques} with $d>v(H)$, then $S$ is safe. 
\end{rem}

We are now ready to state our main technical result, proving the existence of safe \CycleColors{}-determiners and safe \CliqueColors{}-determiners for $q$-tuples consisting of cycles and cliques. 

\begin{theorem}\thlabel{lem:cycle_and_clique_determiners}
Let $\ell\geq 4$, $t\geq 3$, and $q_1,q_2\geq 1$ be integers. Then there exist safe \CycleColors-determiners and safe \CliqueColors-determiners for $\cT(q_1,q_2,\ell,t)$. 
\end{theorem}

Most of Section~\ref{sec:existence_set_signal} is devoted to the proof of~\thref{lem:cycle_and_clique_determiners}. In the same section, we also prove~\thref{lem:cycle_and_clique_senders} below, showing that safe \CycleColors-senders and safe \CliqueColors-senders both exist. 

\begin{theorem}\thlabel{lem:cycle_and_clique_senders}
Let $\ell\geq 4$, $t\geq 3$, and $q_1,q_2\geq 1$ be integers. If $q_1>1$ then there exist safe positive and negative \CycleColors-senders for $\cT(q_1,q_2,\ell,t)$. If $q_2>1$ then there exist safe positive and negative \CliqueColors-senders for $\cT(q_1,q_2,\ell,t)$.
\end{theorem}

\section{Two colors cases}\label{sec:2colors}
Throughout this section the number of colors $q$ is fixed to be $2$, and we drop the color index $q$ in the notation. In this section we determine $s(K_t,T_\ell)$, $s(C_k,C_\ell)$, and $s(K_t,C_\ell)$. We prove that the lower bound in~\eqref{eq:trivial_bounds_sq} is tight for $s(K_t,T_\ell)$ and $s(C_k,C_\ell)$, but not for $s(K_t,C_\ell)$. In the latter two cases, we exemplify the power of the gadget graphs introduced in Section~\ref{sec:preliminaries}. We begin with the case of one clique and one tree.

\begin{proof}[Proof of \thref{thm:2_colors_completetree}]
Let $\ell\geq 2$ and $t\geq3$. First note that the inequality $s(K_t,T_\ell)>t-2$ follows directly from~\eqref{eq:trivial_bounds_sq}. 
For the upper bound, we construct a graph $G$ of minimum degree $t-1$ as follows. Let $H\cong K_{(t-1)(\ell-1)},$ let $F\cong K_{t}$, and let $v$ be a vertex of $F$. For each vertex $u$ of $F-v$, 
create a copy $H_u$ of $H$ on a new set of vertices and identify $u$ with an arbitrary vertex of $H_u$. Note that $d_G(v)=t-1$. We claim that $G\to (K_t,T_\ell)$ while $G-v\nto (K_t,T_\ell)$. For the former, suppose for a contradiction that $\phi$ is a $(K_t,T_\ell)$-free red/blue-coloring of $G.$ Then $\phi$ is $(K_t,T_\ell)$-free on $H_u\cong K_{(t-1)(\ell-1)}$ for every vertex $u$ in $F-v.$ By \cite[Lemma 9]{burr_ramsey-minimal_1982}, there is a unique $(K_t,T_\ell)$-free red/blue-coloring $\widetilde\phi$ of $H_u$, in which the subgraph of blue edges of $H$ is a collection of $(t-1)$ vertex-disjoint cliques, each of size $(\ell-1)$. In particular, in the coloring $\phi$, every vertex $u$ of $F-v$ is incident to a blue copy of $K_{\ell-1}$ in $H_u$. Therefore, every edge of $F$ must be red, creating a monochromatic red copy of $K_t$, a contradiction. 
For the second claim, color the edges of $F-v$ red and use the $(K_t,T_{\ell})$-free coloring $\widetilde \phi$ for every $H_u.$ It is easy to see that this red/blue-coloring of $G-v$ is $(K_t,T_\ell)$-free. 
Thus, any subgraph $G'$ of $G$ that is Ramsey-minimal for $(K_t,T_\ell)$ must contain $v.$ This proves 
$s(K_t,T_\ell)\leq d_G(v)=t-1.$\end{proof}

We now turn our attention to pairs of graphs involving cycles. It follows from~\eqref{eq:trivial_bounds_sq} that $s(C_k,C_\ell)>2$. For $k<\ell$, we now use the existence of safe determiners given by~\thref{lem:siggers_safe_det_cycles} to exhibit a Ramsey-minimal graph for $(C_k,C_\ell)$, with minimum degree three. \thref{thm:2_colors_cyclecycle} then follows by symmetry, since $s(C_k,C_\ell)=s(C_\ell,C_k)$.

\begin{prop}\thlabel{prop:upper_bound_cycleVScycle}
For any $4\leq k<\ell$, we have
\[s(C_k,C_\ell)\leq 3.\]
\end{prop}

\begin{proof}
We construct an appropriate Ramsey-minimal graph. Start with an empty graph on three vertices $\set{x,y,z}$, and between any pair of these vertices add two paths, one of length $k-2$ and one of length $\ell-2$, so that all six paths are internally vertex-disjoint. 
Let $D_r$ and $D_b$ be a safe red- and blue-determiner for $(C_k, C_\ell)$, respectively, as guaranteed by~\thref{lem:siggers_safe_det_cycles}. Attach a copy of $D_r$ to every edge contained in one of the paths of length $k-2$ between $x,y,$ and $z$ and a copy of $D_b$ to every edge contained in one of the paths of length $\ell-2$. Finally, add a new vertex $v$ adjacent to $x$, $y$, and $z$, and call the resulting graph $F$. The construction is illustrated in Figure~\ref{fig:Ex_C4C5} for the case $k=4$ and $\ell=5$, showing only the signal edges for each determiner and the edges incident to $v$. We will now show that $F\to(C_k, C_\ell)$ but $F-v\nto(C_k, C_\ell)$, implying that any subgraph $G$ of $F$ that is Ramsey-minimal for $(C_k, C_\ell)$ has to contain $v$, which in turn proves the proposition. 

Consider an arbitrary red/blue-coloring of $F$. If any copy of $D_r$ or $D_b$ contains a red copy of $C_k$ or a blue copy of $C_\ell$, we are done. Otherwise, by property \ref{axiom:set_determiner_colored_edge} of $D_r$ and $D_b$, the paths of length $k-2$ between the vertices $x,y$, and $z$ must be all red and the paths of length $\ell-2$ between those vertices must be all blue. By the pigeonhole principle, two of the edges incident to $v$ must have the same color; these two edges together with the corresponding red $(k-2)$-path or blue $(\ell-2)$-path then form a red copy of $C_k$ or a blue copy of $C_\ell$. 

\begin{figure}[ht]
    \centering
    \TikzExampleCkCl{4}{4}
    \caption{The graph $F$ in the proof of~\thref{prop:upper_bound_cycleVScycle}}
    \label{fig:Ex_C4C5}
\end{figure}

For the second claim, consider $F-v$ and color each path of length $k-2$ between the vertices $x,y,$ and $z$ red and each path of length $\ell-2$ between those vertices blue. Since $k,\ell>3$, it is easy to see that this partial coloring of $F-v$ is $(C_k, C_\ell)$-free. By property~\ref{axiom:set_determiner_any_color}  of the copies of $D_r$ and $D_b$, we can extend this coloring to the copies of $D_r$ and $D_b$ so that each determiner has a safe $(C_k, C_\ell)$-free coloring. By definition of safeness, this is a $(C_k,C_\ell)$-free coloring of $F-v$.
\end{proof}

Note that the construction requires $k>3$. The case $k=3$ is covered by our next construction, dealing with cliques. To that end, we turn our attention to $s(K_t,C_\ell)$, proving~\thref{thm:2_colors_completecycle}. The idea behind the upper bound construction is very similar to the previous one.

\begin{prop}\thlabel{prop:upper_bound_completeVScycle}
For any integers $t\geq 3$ and $\ell\geq 4$, we have
\[s(K_t,C_\ell)\leq 2(t-1).\]
\end{prop}

\begin{proof}
Let $t\geq 3$ and $\ell\geq 4$. Using the safe determiners from~\thref{lem:siggers_safe_det_clique_cycle}, we construct a graph $G$ that is Ramsey-minimal for $(K_t,C_\ell)$ and satisfies $\delta(G) \leq 2(t-1)$.

We start with the graph $T = K_{2,2,\ldots,2}$, the complete $(t-1)$-partite graph where each independent set contains two vertices. For any pair of vertices in the same class, add a path of length $\ell-2$; as before, all these paths are  vertex-disjoint.
Let $D_r$ and $D_b$ be a safe red- and blue-determiner, respectively, as guaranteed by~\thref{lem:siggers_safe_det_clique_cycle}. Attach a copy of $D_r$ to each edge of $T$ and a copy of $D_b$ to each edge belonging to one of the $t-1$ paths of length $\ell-2$. Add a new vertex $v$ adjacent to all vertices of $T$ and call the resulting graph $F$. This construction is illustrated in Figure~\ref{fig:example_K5_Cl} for $t=5$ and $\ell=5$, showing only the signal edges for each determiners and the edges incident to $v$. As in the proof of~\thref{prop:upper_bound_cycleVScycle}, we will show that 
$F\to (K_t, C_\ell)$ but $F-v\nto (K_t, C_\ell)$.
 
To see the first claim, consider an arbitrary red/blue-coloring of $F$. If any copy of $D_r$ or $D_b$ contains a red copy of $K_t$ or a blue copy of $C_\ell$, then we are done. Hence, all determiners have $(K_t,C_\ell)$-free colorings, forcing the edges of $T$ to be all red and the edges in the $(\ell-2)$-paths connecting pairs of vertices from the same partite set of $T$ to be blue. Now, if both edges between $v$ and one of the vertex classes of $T$ are blue, there is a blue copy of $C_\ell$. Otherwise, there is a red edge from $v$ to each of the $t-1$ partite sets of $T$, resulting in an all-red copy of $K_t$. 
 
\begin{figure}
    \centering
    \TikzExampleKrCl{4}{4} 
    \caption{The graph $F$ in the proof of~\thref{prop:upper_bound_completeVScycle}}
    \label{fig:example_K5_Cl}
\end{figure}
 
For the second claim, color the edges of $T$ red and the edges of the $(\ell-2)$-paths connecting vertices from the same vertex class of $T$ blue. Then, using property~\ref{axiom:set_determiner_any_color} of the copies of $D_r$ and $D_b$, extend this coloring to all determiners so that each one receives a safe $(K_t,C_\ell)$-free coloring. It is easy to see that this gives a $(K_t,C_\ell)$-free coloring of the entire graph $F$. 
\end{proof}

Note that this upper bound for $s(K_t,C_\ell)$ does not match the lower bound from \eqref{eq:trivial_bounds_sq}, as the latter only implies $s(K_t,C_\ell)\geq t+1$. However, \thref{prop:lower_bound_completeVScycle} will prove that our construction does yield the best possible upper bound. We will need an auxiliary lemma, which shows that, if $G$ is a graph on fewer than $2(t-1)$ vertices with no $t$-clique, then there must be at least one vertex common to all $(t-1)$-cliques.

\begin{lemma}\thlabel{lemma:lower_bound_completeVScycle}
Let $t\geq 3$ be any integer and $G$ be a graph on $n<2(t-1)$ vertices with $K_{t-1}\subseteq G$. If 
\[\bigcap_{\substack{H\subseteq G\\H\cong K_{t-1}}} V(H)  = \emptyset \ ,\]
then $K_{t}\subseteq G$.
\end{lemma}

\begin{proof} We proceed by strong induction on $t$. It is easy to check that the  statement is true for $t=3$. Assume now that $t\geq 3$, and suppose the statement to be true up to $t$.

Let $G$ be a graph on $n <2t$ vertices, and let $\mathcal{F}=\set*{H_0,\ldots,H_m}$ be a family of distinct $t$-cliques contained in $G$ whose joint intersection is empty. Suppose additionally that this family is minimal, meaning that every subfamily has a non-empty intersection. Note that we may assume that $m\geq 1$.

Let $S=V(H_1)\cap \ldots \cap V(H_{m})$ be the vertex set in the intersection of the $t$-cliques $H_1,\ldots,H_{m}$ (without considering $H_0$). By the minimality of the family $\mathcal{F}$, we know that $\size*{S}>0$. Further, since $G$ has fewer than $2t$ vertices, it cannot contain two disjoint $t$-cliques. Therefore, as $H_0$ is a $t$-clique and  $S$ is another clique disjoint from $H_0$ in $G$, it follows that $\size{S}<t-1$. Write $\size{S} = t-j$ for some $0<j<k$.

For $i\in[m]$, let $S_i=V(H_i)\setminus S$. Note that each $S_i$ induces a $j$-clique. Each vertex in $S_i$ is adjacent to all vertices in $S$. Therefore, since $\size*{S}=t-j$, if we can find a $(j+1)$-clique in $G\brackets*{\bigcup_{i=1}^m S_i}$, we will have found a $(t+1)$-clique in $G$. We consider two possible cases.

\emph{\underline{Case 1}:} Suppose that $\bigcup_{i=1}^{m} S_i$ has at least $2j$ elements. By definition, both $V(H_0)$ and $\bigcup_{i=1}^{m} S_i$ have empty intersection with $S$, therefore they are both contained in the set $V(G)\setminus S$ whose size is less than $t+j$.  
Since $\size{V(H_0)}=t$ and $\size{\bigcup_{i=1}^{m} S_i} \geq 2j$, they must have at least $j+1$ vertices in common, forming a $(j+1)$-clique in $G\brackets*{\bigcup_{i=1}^{m} S_i}$. 

\emph{\underline{Case 2}:} Assume next that $\bigcup_{i=1}^{m} S_i$ has fewer than $2j$ elements. Then  $G[\bigcup_{i=1}^{m} S_i]$ is a graph on fewer than $2j$ vertices containing a $j$-clique, namely $G[S_1]$. Since $j<t$ and $\bigcap_{i=1}^{m} S_i=\emptyset$, by the induction hypothesis, it follows that $G\brackets*{\bigcup_{i=1}^{m} S_i}$ contains a $(j+1)$-clique. 
\end{proof}

We are now ready to prove a lower bound on $s(K_t, C_\ell)$ from~\thref{lemma:lower_bound_completeVScycle}.~\thref{thm:2_colors_cyclecycle} then follows immediately from~\thref{prop:upper_bound_completeVScycle}.

\begin{prop}\thlabel{prop:lower_bound_completeVScycle}
For any integers $t\geq 3$ and $\ell\geq 4$,  we have
\[s(K_t,C_\ell)\geq 2(t-1).\]
\end{prop}
\begin{proof}
Suppose that $G$ is a Ramsey-minimal graph for $(K_t,C_\ell)$, and let $v$ be a vertex of degree at most $2(t-1)-1$ in $G$, i.e. $\abs{N(v)}<2(t-1)$. By the minimality of $G$, there exists a red/blue-coloring $\phi$ of the edges of $G-v$ with no red copy of $K_{t}$ and no blue copy of $C_\ell$. If $G[N(v)]$ contains no red copy of $K_{t-1}$, then we can extend the coloring $\phi$ to $G$ by coloring all edges incident to $v$ red to obtain a $(K_t,C_\ell)$-free coloring of $G$, a contradiction.

Therefore assume that we have at least one red copy of $K_{t-1}$ in $G[N(v)]$. By~\thref{lemma:lower_bound_completeVScycle}, because $G[N(v)]$ has no red copy of $K_t$ and $\abs{N(v)}<2(t-1)$, there exists at least one vertex $u$ in the intersection of all red copies of $K_{t-1}$ in $G[N(v)]$. Extend $\phi$ to $G$ by coloring the edge $uv$ blue and all other edges from $v$ to $N(v)\setminus\{u\}$ red. This coloring does not create a red copy of $K_t$ and the unique blue edge incident to $v$ cannot create a blue copy of $C_\ell$, again contradicting the fact that $G$ is Ramsey for $(K_t,C_\ell)$. 
\end{proof}

\section{Proof of Theorems~\ref{thm:sq_relation} and~\ref{thm:q1large}}\label{sec:proof_main_result}

As noted earlier, we defer the proofs of \thref{lem:cycle_and_clique_determiners,lem:cycle_and_clique_senders} to Section \ref{sec:existence_set_signal}. In this section, we assume their statements to be true, and use them to prove our main results, \thref{thm:sq_relation,thm:q1large}. Recall that $\cT=\cT(q_1,q_2,\ell, t)$ denotes the $q$-tuple of cycles and cliques as defined in \eqref{eq:definition_cT}, and that \CycleColors{} and \CliqueColors{} denote the cycle-colors $\{1,\ldots,q_1\}$ and clique-colors $\{q_1+1,\ldots,q\}$, respectively, while \AllColors{} denotes the full color palette $\{1,\ldots,q\}$. The idea is to express our function $s_q(\cT)$ in a different way, through a certain packing parameter. This idea was first formalized in~\cite{fox2016minimum} in their study of $s_q(K_t)$ in the multicolor setting, but, as the authors of~\cite{fox2016minimum} note, this idea can already be found implicitly in~\cite{burr_graphs_1976}.

\subsection{Packing parameters}\label{subsec:passing_to_packing}
In this section we generalize the packing parameter defined in~\cite{fox2016minimum}. A {\em color pattern} on vertex set $V$ is a collection of edge disjoint graphs $G_1,\ldots,G_m$ on the same vertex set $V$. A color pattern is $H$-free if every graph in it is $H$-free. 

\begin{definition}\thlabel{def:pq1q2t}
Given positive integers $t\geq 2$ and $q_1,q_2\geq 0$, let $P_{q_1,q_2}(t)$ be the smallest integer $n$ such that there exists a color pattern 
$G_{q_1+1},\ldots,G_{q_1+q_2}$ on vertex set $[n]$ such that  
\begin{enumerate}[label=(P\arabic*)]
    \item the graph $G_j$ is $K_{t+1}$-free for every $j\in \CliqueColors{}$, and \label{item:Packing_P1}
    \item for every vertex-coloring $\lambda:[n]\to\AllColors{}$, we have that $(a)$
    two distinct vertices $u$ and $w$ receive the same cycle-color; or $(b)$ 
    there exists a clique-color $j\in \CliqueColors{}$ such that $G_j$ contains a copy of $K_{t}$ on the vertices of color $j.$
    \label{item:Packing_P2}
\end{enumerate}
\end{definition}

For $q_1=0$, this parameter was introduced in~\cite{fox2016minimum}, and for all $q_2\geq 2$ and $t\geq 3$,Theorem 1.5 in~\cite{fox2016minimum} establishes that $s_{q_2}(K_t) = P_{0,q_2}(t-1)$. The following lemma generalizes this theorem and proves that $s_q(\cT)$ does not depend on $\ell$.

\begin{lemma}\thlabel{lem:packing_equivalence}
For all integers $\ell\geq 4$, $t \geq 3$, and $q_1, q_2\geq 0$, we have
\[ s_{q_1+q_2}(\cT(q_1,q_2,\ell,t)) = P_{q_1,q_2}(t-1).\]
\end{lemma}

\begin{proof}
Set $q=q_1+q_2$ and $\cT = \cT(q_1,q_2,\ell,t)$. We divide the proof into two claims. 

\begin{clm}
$ s_q(\cT) \leq P_{q_1,q_2}(t-1).$
\end{clm}

\begin{proof}
As explained previously, in this proof we assume the existence of gadget graphs as guaranteed by \thref{lem:cycle_and_clique_determiners,lem:cycle_and_clique_senders}. Let $n=P_{q_1,q_2}(t-1)$ and $G_{q_1+1},\ldots, G_{q}$ be a color pattern on $[n]$ that satisfies~\ref{item:Packing_P1} and~\ref{item:Packing_P2}. For every pair of distinct vertices $u,w\in [n]$ and every cycle-color $i\in \CycleColors{}$, add a path $P_i(u,w)$ of length $\ell-2$ between $u$ and $w$ such that the internal vertices of these paths are pairwise disjoint.  Finally add a new vertex $v$, and connect it to each vertex in $[n]$. Call the resulting graph $H$.\smallskip

Assume first that $q_1,q_2>1$. Now, let $S^+_c$ and $S^-_c$ be a safe positive and negative \CycleColors-sender for $\cT$, respectively, and let $S^+_k$ and $S^-_k$ be a safe positive and negative \CliqueColors-sender for $\cT$; all of these gadgets exist by~\thref{lem:cycle_and_clique_senders}. Let $E=\{e_1,\ldots,e_{q}\}$ be a matching of size $q$. For each pair $i, j\in \CycleColors{}$ of distinct cycle-colors, join the edges $e_i$ and $e_j$ by a copy of $S^-_c$. Similarly, for each pair $i, j\in \CliqueColors{}$ of distinct clique-colors, join the edges $e_i$ and $e_j$ by a copy of $S^-_k$. For every clique-color $i\in\CliqueColors{}$ and every edge $f\in E(G_i)$, join the edges $e_i$ and $f$ by a copy of $S^+_k$. Then for each $i\in \CycleColors$ and for each edge $f\in P_i(u,w)$, join the edges $e_i$ and $f$ by a copy of $S^+_c$. Call the resulting graph $G$.

We will show that $G\to_q \cT$ but $G-v \nto_q \cT$. We begin with the latter. For this we define a $\cT$-free coloring. For all $i\in \CliqueColors$, give all edges of $G_i$ color $i$. For all $i\in \CycleColors$ and every pair of distinct vertices $u,w\in [n]$, color the edges of $P_i(u,w)$ with color $i$. Finally, for all $i\in [q]$, give $e_i$ color $i$. This coloring can now be extended to the set-senders so that each set-sender receives a safe $\cT$-free coloring. Suppose there exists a monochromatic cycle in a cycle-color or clique in a clique-color. By the safeness of the coloring of each set-sender, we know that such a monochromatic subgraph has to be contained in $H-v$. But $H-v$ contains no monochromatic copy $K_t$ in a clique-color by property~\ref{item:Packing_P1} of the color pattern. By construction, it is not difficult to see that it also contains no monochromatic copy of $C_\ell$ in a cycle-color. Hence, this is a $\cT$-free coloring of $G-v$, as claimed.\smallskip

We now prove that $G\to_q \cT$. For the sake of contradiction, let $\phi:E(G)\to\AllColors{}$ be a $\cT$-free $q$-coloring of the edges of $G$. In any such coloring, property~\ref{axiom:set_sender_colored_edges} of the copies of $S^-_c$ and $S^-_k$ ensures that $\set*{\phi(e_1),\ldots,\phi(e_{q_1})} = \CycleColors$, while $\set*{c(e_{q_1+1}),\ldots,c(e_{q})} = \CliqueColors{}$. Without loss of generality, we may assume that for any $i\in\AllColors{}$, we have $\phi(e_i)=i$. Property~\ref{axiom:set_sender_colored_edges} of the copies of $S^+_k$ and $S^+_c$ further ensures that, for any $i\in\CliqueColors{}$, each edge in $G_i$ has color $i$, and for each pair of vertices $u,w\in[n]$ and each $j\in\CycleColors{}$, the edges of $P_j(u,v)$ receive color $j$.

Consider now the edges from $v$ to $N(v)=[n]$. These induce a natural vertex-coloring $\lambda:[n]\to\AllColors{}$ defined by $\lambda(u)=\phi(vu)$ for each $u\in[n]$. Then by property~\ref{item:Packing_P2}, it follows that either there are two distinct vertices $u,w\in[n]$ such that $\lambda(u) = \lambda(w) = j$ for some $j\in\CycleColors{}$, or there exists a clique-color $j\in \CliqueColors{}$ such that $G_j[\lambda^{-1}(\set{j})]$ contains a copy of $K_{t-1}$. In the former case $P_j(u,w)$  forms a monochromatic copy of $C_{\ell}$ in color $j$ together with $v$. In the latter case, the copy of $K_{t-1}$ forms a monochromatic copy of $K_t$ in color $j$ together with $v$. 

It follows that $G$ is $q$-Ramsey for $\cT$, while $G-v$ is not. So any $q$-Ramsey-minimal subgraph of $G$ must contain the vertex $v$, and therefore  $s_q(\cT) \leq  d_{G}(v) = n = P_{q_1,q_2}(t-1)$.\smallskip

If $q_1=1$ and/or $q_2=1$, we use a safe \CycleColors-determiner $D_c$ instead of \CycleColors-senders, and/or a safe \CliqueColors-determiner $D_k$ instead of \CliqueColors-senders. These gadgets exist by~\thref{lem:cycle_and_clique_determiners}. If $q_1=1$, for each $i\in \CycleColors$ and for each edge $f\in P_i(u,w)$, we attach a copy of $D_c$ to $f$. If $q_2=1$, for each $i\in \CliqueColors$ and every edge $f\in E(G_i)$, we attach a a copy of $D_k$ to $f$. The rest of the proof is identical to the case $q_1,q_2>1$, using corresponding properties of set-determiners.\end{proof}

\begin{clm}
$s_q(\cT) \geq P_{q_1,q_2}(t-1).$
\end{clm}

\begin{proof}
Towards a contradiction, assume that there exists a graph $G$ with a vertex $v$ of degree $n < P_{q_1,q_2}(t-1)$, such that $G$ is $q$-Ramsey-minimal for $\cT$. By minimality, there exists a $\cT$-free $q$-coloring $\phi$ of the edges of $G-v$. This coloring induces a color pattern $G_{q_1+1},\ldots, G_{q}$ on $N(v)$, corresponding to the colors $q_1+1, \dots, q$ respectively, such that every $G_j$ is $K_t$-free. Since $\size{N(v)}<P_{q_1,q_2}(t-1)$ and each $G_j$ is $K_t$-free, by property~\ref{item:Packing_P2} there must exist a vertex-coloring  $\lambda:N(v)\to\AllColors{}$ such that no two vertices in $N(v)$ receive the same cycle-color and there is no clique-color $j$ such that $G_j[\lambda^{-1}(\set{j})]$ contains a copy of $K_{t-1}$. Now, we extend $\phi$ to all of $G$ by setting $\phi(uv)=\lambda(u)$ for each $u\in N(v)$.

By the properties of $\lambda$, this extended coloring has no monochromatic copy of $C_{\ell}$ in any color $j\in \CycleColors{}$ and no monochromatic copy of $K_t$ in any color $j\in \CliqueColors{}$, contradicting the fact that $G$ is $q$-Ramsey for $\cT$. 
\end{proof}
\let\qed\relax
\end{proof}

\subsection{Proof of Theorem~\ref{thm:sq_relation}}\label{subsec:proof_sq_relation}
We are now ready to prove our first main result in the multicolor setting. We begin with the lower bound. 

\begin{lemma}\thlabel{clm:new_lower_bound}
    For all $q_1,q_2\geq 1$, $t\geq 3$, and $\ell\geq 4$, we have
    \begin{align}\label{eq:lowerbound}
        s_{q_1+q_2}(\cT(q_1,q_2,\ell, t))\geq s_{q_2}(K_t)+s_{q_1}(C_\ell)-1 = s_{q_2}(K_t)+q_1.
    \end{align}
\end{lemma}
\begin{proof}
    Set $q=q_1+q_2$ and $\cT = \cT(q_1,q_2,\ell,t)$, and suppose that $G$ is a $q$-Ramsey-minimal graph for $\cT$ containing a vertex $v$ of degree at most $s_{q_2}(K_t)+s_{q_1}(C_\ell)-2$. Let $\phi:E(G-v) \rightarrow [q]$ be a $\cT$-free $q$-coloring of $G-v$. Let $G'$ be the subgraph of $G$ containing all edges of $G-v$ with colors $q_1+1,\dots, q$ and any set of $\min\set{s_{q_2}(K_t)-1,\ \deg_G(v)}$ edges of $G$ incident to $v$. We know that $G'-v$ is not $q_2$-Ramsey for $K_t$, and since $\deg_{G'}(v) < s_{q_2}(K_t)$, it follows that $G'$ itself cannot be $q_2$-Ramsey for $K_t$. Thus, we can recolor the edges of $G'$ using colors $q_1+1,\dots, q$ so that there is no monochromatic copy of $K_t$ inside. Now, we can apply the same argument to $G-G'$ to obtain a $C_\ell$-free coloring of it with the colors $1,\dots, q_1$. These two colorings together yield a $\mathcal{T}$-free coloring of $G$, a contradiction. The last equality follows from the fact $s_q(C_\ell)=q+1$ \cite{boyadzhiyska2020minimal}.
\end{proof}

From the proof of this lower bound it becomes clear that this is actually a generalization of the trivial lower bound given in~\eqref{eq:trivial_bounds_sq}. We now proceed with the upper bound. For this we take a slightly indirect approach: instead of working directly with the parameter $s_q$, we show a relation between the two packing parameters.

\begin{lemma}
For all $q_1,q_2\geq 1$, $t\geq 3$, and $\ell\geq 4$, we have
    \begin{align}\label{eq:upperbound}
        s_{q_1+q_2}(\cT(q_1,q_2,\ell, t)) = P_{q_1, q_2} (t-1) \leq P_{0,q_1+q_2}(t-1) = s_{q_1+q_2}(K_t).
    \end{align}
\end{lemma}
\begin{proof}
    Again set $q=q_1+q_2$ and let $n = P_{0,q}(t-1)$. Let $G_1,\dots, G_q$ be a color pattern on $[n]$, as guaranteed by~\thref{def:pq1q2t} of $P_{0,q}(t-1)$. Consider only the last $q_2$ graphs; we claim that this color pattern satisfies properties~\ref{item:Packing_P1} and~\ref{item:Packing_P2} from~\thref{def:pq1q2t} of $P_{q_1, q_2}(t-1)$. The first property is clear. Now let $\lambda:[n]\to\AllColors$ be any coloring. Then we know that there is some $j\in \AllColors$ such that $G_j$ contains a monochromatic copy of $K_{t-1}$ on the vertices of color $j$. Now, if $j>q_1$, then case $(b)$ from property~\ref{item:Packing_P2} occurs. Otherwise, we have $j\leq q_1$, and thus there must be at least $t-1\geq 2$ vertices of color $j$, implying that case $(a)$ from property~\ref{item:Packing_P2} happens. Hence $P_{q_1, q_2} (t-1) \leq P_{0,q}(t-1)$, and the two equalities follow from \thref{lem:packing_equivalence} and the discussion that preceeds it.
\end{proof}

\subsection{Proof of Theorem~\ref{thm:q1large}}\label{subsec:proof_q1_large}

We now prove our second main result for multiple colors. In~\cite[Lemmas 4.2 and 4.4]{fox2016minimum}, it was shown that, for all $q \geq 2$ and $t \geq 3$, there exists a color pattern $G_1,\ldots,G_q$ on the vertex set $[n]$, for some $n$, such that
\begin{enumerate}[label=(\roman*)]
    \item $G_i$ is $K_t$-free for every $i \in [n]$, and
    \item any subset of $[n]$ of size $n/q$ contains a copy of $K_{t-1}$ in each color.
\end{enumerate}
The results in~\cite{fox2016minimum} include bounds on $n$ in terms of $q$, which are unnecessary for our purpose. \thref{thm:q1large} follows from the next lemma by taking $\eps \to 0$. 

\begin{lemma}
Given $0< \varepsilon < 1$ and integers $q_2\geq 1$ and $t\geq 3$, there exists an integer $q_0\geq 1$ such that, for all $q_1\geq q_0$, we have
\[P_{q_1,q_2}(t-1)\leq (1+\varepsilon)q_1.\]
\end{lemma}
\begin{proof}
Let $0< \varepsilon < 1$, $q_2\geq 1$, and $t\geq 3$ be fixed. For $q_1$ large enough, there exists a color pattern $G_1,\dots, G_{q^\ast}$ on $n \in [(1+\eps/2)q_1, (1+\eps)q_1]$ vertices, given by the result in~\cite{fox2016minimum}, with $q^\ast$ large enough compared to $q_2$.

Keeping only the first $q_2$ graphs in the color pattern, which we denote for convenience by $G_{q_1+1}, \dots, G_{q_1+q_2}$, we claim that they satisfy properties~\ref{item:Packing_P1} and~\ref{item:Packing_P2}. The first one is clear. For the second one, consider a vertex coloring $\lambda: [n] \to[q]$, where $q=q_1+q_2$. Let $\mathcal{C}$ be its largest color class in $\CliqueColors$, with color $c$. If $(a)$ does not hold, by the pigeonhole principle the color class $\mathcal{C}$ has size at least $\frac{n-q_1}{q_2}$. Since $q^\ast$ is large enough compared to $q_2$, and by choice of $n$, we have $\frac{n-q_1}{q_2} \geq \frac{n}{q^\ast}$. By property (ii) above, we know that there exists a copy of $K_{t-1}$ in $G_{c}[\mathcal{C}]$. Therefore if $(a)$ of~\ref{item:Packing_P2} does not hold then $(b)$ does, and $P_{q_1,q_2}(t-1) \leq n\leq (1+\eps) q_1$.
\end{proof}


\section{Existence of set-determiners and set-senders}\label{sec:existence_set_signal}

In this section we construct set-determiners and set-senders for tuples of the form {$(C_\ell,\dots, C_\ell, K_t, \allowbreak \dots, K_t)$}, that is, we prove~\thref{lem:cycle_and_clique_determiners,lem:cycle_and_clique_senders}. Our set-senders will be constructed in several stages. Before diving into the proofs, we give a brief overview.

Throughout the rest of the section, assume that $\ell\geq 4$, $t\geq 3$, and $q, q_1,q_2\geq 1$ are fixed integers such that $q_1+q_2=q$ and recall that $\cT=\cT(q_1,q_2,\ell,t)$ denotes the $q$-tuple of cycles and cliques as defined in \eqref{eq:definition_cT}. First, we construct a graph $\Gamma$ that is $q$-Ramsey for the tuple $\cT$ and has certain special properties; for this, we generalize the ideas of Bollob{\'a}s, Donadelli, Kohayakawa, and Schelp~\cite{bollobas2001ramsey} used to construct $2$-Ramsey graphs for certain pairs of graphs, including $(C_\ell, K_t)$, to multiple colors. This graph $\Gamma$ is built by sampling a random hypergraph, applying alterations to remove all short cycles from it, and then replacing every hyperedge by a large (depending only on $t$) clique. In order to prove the claimed properties of $\Gamma$, we use a number of results, all of which are fairly standard by now. Second, we modify $\Gamma$ slightly and construct set-determiners for each of the color palettes $\cS(C_\ell)$ and $\cS(K_t)$. This is a generalization of a construction given by Siggers in~\cite{siggers_non-bipartite_2014}, valid for certain pairs of the form $(C_\ell, H)$. Finally, since we need finer control over the color patterns we force on given set of edges when $q_1>1$ or $q_2>1$, we build set-senders from our set-determiners. This final step is the main novelty in this section. 

\subsection{Preliminary results}
We begin by collecting the different results that will be needed for the construction and proof of the claimed properties of the graph $\Gamma$. 

\medskip 
\noindent\textbf{Hypergraphs with few short cycles.} 
First, we need to construct a uniform hypergraph with no short cycles that is nevertheless not too sparse. This is done using a standard construction due to Erd\H{o}s and Hajnal~\cite{erdos_chromatic_1966}, starting from a random hypergraph. We state the necessary results about random hypergraphs without proof, as these are by now standard applications of the probabilistic method. A \emph{cycle of length $s$} in a hypergraph $\mathcal{H}$ is a sequence $e_1,v_1,e_2,v_2\dots,e_s,v_s$ of distinct hyperedges and vertices of $\mathcal{H}$ such that $v_i\in e_i\cap e_{i+1}$ for all $1\leq i< s$ and $v_s \in e_s\cap e_1$. Note in particular that two edges intersecting in more than one vertex form a cycle of length two in $\mathcal{H}$. The \emph{girth} of a hypergraph $\mathcal{H}$ is the length of the shortest cycle in $\mathcal{H}$ (if no cycle exists, then by convention we say that the girth of $\mathcal{H}$ is infinity).

\begin{lemma}\thlabel{lem:randomhypergraph}
Let $\ell, h\geq 2$ be fixed integers and $p_h = A n^{-(h-1)+1/(\ell-1)}$, where $A$ is a constant. For an integer $n\geq 1$, let $\mathcal{H}_{n,p_h}$ be a random $h$-uniform hypergraph on $[n]$ in which each $h$-subset of $[n]$ is added as an edge with probability $p_h$, independently of all other $h$-subsets. Then, as $n\to\infty$, the following hold with high probability:
\begin{enumerate}[label=(\roman*)]
    \item $e(\mathcal{H}_{n,p_h}) = (1+o(1))\binom{n}{h}p_h$.\label{lem:randomhypergraph:edges}
    \item The number of cycles in $\mathcal{H}_{n,p_h}$ of length less than $\ell$ is $o(e(\mathcal{H}_{n,p_h}))$.\label{lem:randomhypergraph:cycles}
\end{enumerate}
\end{lemma}

Part~\ref{lem:randomhypergraph:edges} follows from an application of the Chernoff bound (see for example~\cite[Theorem~2.1]{mcdiarmid1998concentration}), while part~\ref{lem:randomhypergraph:cycles} is shown using a first-moment argument. 

\medskip
\noindent\textbf{Quantitative version of Ramsey's theorem.}
The following lemma is a simple consequence of Ramsey's theorem and is obtained by a straightforward averaging argument. Informally, it says that, for any $r$-tuple of graphs $(H_1,\dots, H_r)$, if we $r$-color a sufficiently large complete graph, then we can find not just one monochromatic $H_i$ in the correct color, but many of them. The proof is a simple generalization of the one given, for example, in~\cite[Theorem 2]{nenadov_short_2016}.

\begin{lemma}[Quantitative version of Ramsey's theorem]\thlabel{lem:Ramsey}
Let $r\geq 1$ and $H_1,\dots, H_r$ be graphs. Then there exist a real number $c = c(H_1,\dots, H_r) > 0$  and an integer $k_0 = k_0 (H_1,\dots, H_r) \geq 1$ such that, if $k\geq k_0$ and the edges of $K_k$ are colored with $r$ colors, then there exists an $i\in[r]$ such that there are at least $c k^{v(H_i)}$ monochromatic copies of $H_i$ in color $i$.
\end{lemma}

\medskip
\noindent\textbf{Colorful sparse regularity lemma.} One of the tools required to show that $\Gamma$ is $q$-Ramsey for the tuple $\cT$ is a version of \Szemeredi{}'s celebrated regularity lemma~\cite{szemeredi_regular_1975}. More specifically, we will need the colorful sparse version of the lemma, as given for example in \cite{letzter2016path} (see also \cite[Lemma 3.1]{haxell_induced_1995}). Before giving the precise statement in~\thref{lem:regularity} below, we again need several definitions.

\begin{definition}
Let $G$ be a graph on $n$ vertices, $0<\eta \leq 1$, and $0\leq p\leq 1$. Also let $U$ and $W$ be disjoint subsets of $V(G)$.
The \emph{$p$-density} of the pair $(U,W)$ is defined to be \[d_{G,p}(U,W) = \frac{e_G(U,W)}{p|U||W|},\]
where $e_G(U,W)$ denotes the number of edges in $G$ with one endpoint in $U$ and one endpoint in $W$.

The pair $(U,W)$ is said to be \emph{$(\eps, p)$-regular} if, for all $U'\subseteq U$ and $W'\subseteq W$ with $|U'|\geq \eps |U|$ and $|W'|\geq \eps|W|$, we have
\[|d_{G,p}(U',W')-d_{G,p}(U,W)| \leq \eps.\]
If $(U,W)$ is $(\eps, p)$-regular with $p = \frac{e_G(U,W)}{|U||W|}$, then we say that $(U,W)$ is \emph{$(\eps)$-regular} for short.
A partition $P = (V_1,\dots, V_k)$ of $V(G)$ is an \emph{equipartition} if $|V_i| \in \set*{\floor*{\frac{v(G)}{k}},\ceil*{\frac{v(G)}{k}}}$ for all $i\in [k]$. An equipartition is said to be an \emph{$(\eps, p)$-regular partition} if all but at most $\eps \binom{k}{2}$ pairs $(V_i,V_j)$ are $(\eps,p)$-regular.

\smallskip
A graph $G$ is said to be \emph{$(\eta, D, p)$-upper uniform} if, for all disjoint $U,W\subseteq V(G)$ with $|U|,|W|\geq \eta v(G)$, we have $d_{G,p}(U,W) \leq D.$

\end{definition}

We are now ready to state the version of the regularity lemma that we are going to use.

\begin{lemma}[Colorful sparse regularity lemma]\thlabel{lem:regularity}

Let $\eps >0$ and $D>1$ be fixed reals and $k_0\geq 1$ and $r\geq 1$ be integers. Then there exist constants $\eta = \eta(\eps, k_0, D, r)$ and $K_0 = K_0(\eps, k_0, D, r)$ for which the following holds: If $0\leq p\leq 1$ and $G_1,\dots, G_r$ are $(\eta, D, p)$-upper uniform graphs on vertex set $[n]$, then there is an equipartition $(V_1,\dots, V_k)$ of $[n]$ for some $k_0\leq k\leq K_0$ such that all but at most $\eps\binom{k}{2}$ of the pairs $(V_i,V_j)$ are $(\eps, p)$-regular in  $G_s$ for all $s\in [r]$.
\end{lemma}

We will also need the following additional technical lemma, which can be found for example in~\cite[Lemma~4.3]{gerke2005sparse}.
\begin{lemma}\thlabel{lem:feweredges}
Given $0< \eps < 1/6$, there exists a constant $\beta>0$ such that the following holds. For any graph $F = (V_1\cup V_2, E)$ where the pair $(V_1,V_2)$ is $(\eps)$-regular in $F$, and for all $M$ satisfying $\beta v(F) \leq M \leq e(F)$, there exists a subgraph $F'=(V_1\cup V_2, E')$ with $|E'| = M$ and such that $(V_1,V_2)$ is $(2\eps)$-regular in $F'$.
\end{lemma}

\medskip
\noindent\textbf{Enumeration lemma for $C_\ell$-free graphs.}
Let $m,M\geq 1$ and $\ell\geq 4$ be integers, and let $\eps>0$ be a real number. Let $V_1,\dots, V_\ell$ be disjoint sets, each of size $m$. Let $\mathcal{G}(\ell,m, (V_i)_{i=1}^\ell,M,\eps)$ denote the collection of graphs $G$ such that
\begin{itemize}
    \item $V(G) = V_1\cup\dots\cup V_\ell$, where $|V_i| = m$ for each $i\in [\ell]$,
    \item each $V_i$ is an independent set in $G$,
    \item the pair $(V_i,V_{i+1})$ is $\parens*{\eps, \frac{M}{m^2}}$-regular in $G$ with $e_G(V_i,V_{i+1})=M$ for all $i\in [\ell]$\footnote{For convenience, we define $V_{\ell+1}  = V_1$}, and
    \item there are no edges between any other pair $(V_i,V_j)$.
\end{itemize}
 
In other words, the graphs in $\mathcal{G}(\ell,m, (V_i)_{i=1}^\ell,M,\eps)$ are blow-ups of the cycle $C_\ell$ in which each vertex $v_i$ of $C_\ell$ is blown-up to an independent set $V_i$ of size $m$ and such that each edge $v_iv_{i+1}$ of $C_\ell$ corresponds to an $\parens{\eps, \frac{M}{m^2}}$-regular pair $(V_i,V_{i+1})$. Let $\mathcal{F}(\ell,m, (V_i)_{i=1}^\ell,M,\eps)$ denote the set of graphs in $\mathcal{G}(\ell,m, (V_i)_{i=1}^\ell,M,\eps)$ that do not contain $C_\ell$ as a subgraph.

The following enumeration lemma was shown by Gerke, Kohayakawa, R\"odl, and Steger~\cite[Theorem 5.2]{gerke2007small}; it is a special case of a well-known conjecture by Kohayakawa, \L{}uczak, and R\"odl~\cite{klr1997} (the so-called K\L{}R conjecture), which was famously resolved in the general case using the container method~\cite{balogh2015independent,saxton2015hypergraph}.

\begin{lemma}[Counting Lemma]\thlabel{lem:counting}
For any real number $\alpha > 0$ and integer $\ell\geq 4$, there are constants $\eps_0 = \eps_0(\ell, \alpha) >0, C_0 = C_0(\ell, \alpha) > 0,$ and $m_0 = m_0(\ell, \alpha)\geq 1$ such that, for all $m\geq m_0$, $0<\eps\leq\eps_0$, and $M\geq C_0m^{1+1/(\ell-1)}$, we have
\begin{align*}
    \abs{\mathcal{F}\parens{\ell,m, (V_i)_{i=1}^\ell,M,\eps}}\leq \alpha^M\binom{m^2}{M}^\ell.
\end{align*}
\end{lemma}

\subsection{Construction of a special graph \texorpdfstring{$\mathbf{\Gamma}$}{Gamma}}

For the rest of the section, assume that $n$ is a sufficiently large integer with respect to $\ell, t, q, q_1,$ and $q_2$; in all asymptotic estimates in this section, we assume that $n$ tends to infinity. We begin by fixing some constants. 
Let $h = r_{q_2}(K_t)$; it is not difficult to check that $K_h$ is minimal $q_2$-Ramsey for $K_t$.  Let
\begin{align*}
k_0 = k_0(\underbrace{C_\ell, \dots, C_\ell}_{q_1\text{ times}}, K_h,K_2), \quad c = c(\underbrace{C_\ell, \dots, C_\ell}_{q_1\text{ times}}, K_h,K_2)
\end{align*}
be the constants given by~\thref{lem:Ramsey}. We next set
\begin{align*}
\rho = \frac{c}{2q_1},\quad \alpha = \frac{\rho^\ell}{e^{\ell+1}},\quad D=3h^2.
\end{align*}
Let $\eps_0 = \eps_0(\ell, \alpha), m_0 = m_0(\ell, \alpha)$, and $C_0 = C_0(\ell, \alpha)$ be the constants given by~\thref{lem:counting}, and set 
\begin{align*}
\eps = \min\set{\rho\eps_0/2, \rho/10},\quad C = \max\set{C_0,1}.
\end{align*}
Further, let
\begin{align*}
    \eta = \eta(\eps, k_0, D, q_1),\quad    K_0 = K_0(\eps, k_0, D, q_1),\quad    \beta = \beta(\eps/\rho)
\end{align*}
be the constants from~\thref{lem:regularity,lem:feweredges}.  Finally, define
\begin{align*}
    A = \max\set{(h+1)e^{-h}, \rho^{-1}K_0^{1-1/(\ell-1)}C},\quad p_h = An^{-(h-1)+1/(\ell-1)}, \quad p_e = An^{-1+1/(\ell-1)}.
\end{align*}

\smallskip
Let $\mathcal{H}$ be a hypergraph on $[n]$ sampled from $\mathcal{H}_{n,p_h}$ as in~\thref{lem:randomhypergraph}. Let $\mathcal{G}$ be the hypergraph obtained from $\mathcal{H}$ after the removal of one hyperedge from each cycle of length less than $\ell$. Then $\mathcal{G}$ contains no cycles of length less than $\ell$; by~\thref{lem:randomhypergraph}~\ref{lem:randomhypergraph:edges} and~\ref{lem:randomhypergraph:cycles}, we also know that $e(\mathcal{G}) = (1+o(1))\binom{n}{h}p_h$.

Let $\Gamma$ be the graph on $[n]$ obtained by embedding a copy of $K_h$ into every hyperedge of $\mathcal{G}$, i.e., $\Gamma$ is the graph on $[n]$ in which two vertices are adjacent if and only if they are contained in a common hyperedge of the hypergraph $\mathcal{G}$. The main difference between this construction and the one given in~\cite{bollobas2001ramsey} is that, in order to deal with multiple colors, instead of placing just a copy of our target graph $K_t$ in each hyperedge of $\mathcal{G}$, we place a Ramsey graph for it. For a given graph $F$ and a subgraph $\Gamma'\subseteq\Gamma$, we call a copy $F'$ of $F$ in $\Gamma'$ a \emph{hyperedge copy} if the vertex set of  $F'$ is contained within a single hyperedge of $\mathcal{G}$. All remaining copies of $F$ in $\Gamma'$ are referred to as \emph{non-hyperedge copies}. In addition, we call a subgraph $\Gamma'\subseteq \Gamma$ \emph{transversal} if there exists a bijection $f:E(\Gamma') \to E(\mathcal{G})$ such that $e\subseteq f(e)$ for all $e\in E(\Gamma')$; that is, $\Gamma'$ is transversal if it contains exactly one edge from each hyperedge copy of $K_h$ in $\Gamma$.

\smallskip
Before showing that with high probability $\Gamma \rightarrow_q \cT(q_1,q_2,\ell, t)$ in~\thref{thm:GammaRamsey}, we discuss some properties of the graph $\Gamma$ in~\thref{lem:propertiesGamma} below. The proofs of parts~\ref{lem:propertiesGamma:only_hyperedge_copies},
~\ref{lem:propertiesGamma:upper_uniform},~\ref{lem:propertiesGamma:copies_of_H} are essentially the same as those given in~\cite{bollobas2001ramsey}. The proof of~\ref{lem:propertiesGamma:contains_cycle} is by now also standard in light of the recently resolved K\L{}R conjecture; as we believe that our version (using more modern results) can be generalized more easily to other tuples of graphs, we include the details in Appendix \ref{app:proof_lemma}.

\begin{lemma}\thlabel{lem:propertiesGamma}
The graph $\Gamma$ satisfies each of the following properties with high probability:
\begin{enumerate}[label=(\alph*)]
    \item If $F$ is a $2$-connected graph with no induced cycles of length $\ell$ or more, then every copy of $F$ in $\Gamma$ is a hyperedge copy; in particular, every copy of $K_h, K_t$, and $C_{\ell'}$ for any $\ell'< \ell$ in $\Gamma$ is a hyperedge copy.\label{lem:propertiesGamma:only_hyperedge_copies}
    \item $\Gamma$ is $(\eta, D, p_e)$-upper uniform.\label{lem:propertiesGamma:upper_uniform}
    \item Let $m$ be an integer satisfying  $\frac{n}{K_0}\leq m \leq \frac{n}{k_0}$, let $(V_1,\dots, V_\ell)$ be any $\ell$-tuple of disjoint subsets of $V(\Gamma)$ such that $|V_i| = m$ for all $i\in[\ell]$, and let $\Gamma'\subseteq \Gamma$ be transversal. If the pairs $(V_i,V_{i+1})$ are $(\eps, p_e)$-regular in $\Gamma'$ with $p_e$-density at least $\rho$ for all $i\in [\ell]$, then $\Gamma'[V_1\cup \dots \cup V_\ell]$ contains a copy of $C_\ell$.\label{lem:propertiesGamma:contains_cycle}
    \item Let $m$ be an integer satisfying $\frac{n}{\log n} \leq m\leq \frac{n}{h}$ and  $(W_1,\dots, W_h)$ be an $h$-tuple of pairwise disjoint subsets of $V(\Gamma)$ with $|W_i| = m$ for all $i\in [h]$. Then there are at least $\frac{1}{4}m^hp_h$ distinct copies of $K_h$ contained in the multipartite subgraph of $\Gamma$ spanned by $W_1\cup\dots\cup W_h$.\label{lem:propertiesGamma:copies_of_H}
\end{enumerate}
\end{lemma}

We are now ready to show the main result of this section.

\begin{theorem}\thlabel{thm:GammaRamsey}
With high probability, $\Gamma \rightarrow_q\cT$.
\end{theorem}
\begin{proof} 
We condition on $\Gamma$ having all of the properties given in~\thref{lem:propertiesGamma}. For convenience, we may assume also that $\frac{n}{k}$ is an integer for all $k_0 \leq k\leq K_0$. Consider an arbitrary $q$-coloring $\phi$ of the graph $\Gamma$. If any copy of $K_h$ receives only colors in $\CliqueColors$, then we are done since $h = r_q(K_t)$. So suppose that each such copy has at least one edge whose color comes from $\CycleColors$. Let $\Gamma'$ be a graph on $V(\Gamma) = [n]$ obtained by taking exactly one edge that has a cycle-color from each hyperedge copy of $K_h$ in $\Gamma$; note that $\Gamma'$ is a transversal subgraph. We claim that $\Gamma'$ contains a copy of $C_\ell$ in some cycle-color. 

For each $s\in \CycleColors$, let  $G_s$ be the subgraph of $\Gamma'$ on vertex set $[n]$ consisting of all edges that have color $s$ under $\phi$. By~\thref{lem:propertiesGamma}~\ref{lem:propertiesGamma:upper_uniform}, we know that $\Gamma$ is $(\eta, D, p_e)$-upper uniform, and hence $G_s$  is $(\eta, D, p_e)$-upper uniform for all $s\in\CycleColors$. So by~\thref{lem:regularity}, there exists an equipartition $(V_1,\dots, V_k)$ of $[n]$ in which all but at most $\eps\binom{k}{2}$ pairs $(V_i,V_j)$ are $(\eps, p_e)$-regular in every $G_s$ for $s\in \CycleColors$. Let $m = \frac{n}{k}$; by our choice of $k_0, K_0$, and $n$, we know that $m$ is an integer and $\frac{n}{K_0} \leq  \frac{n}{k} = m \leq \frac{n}{k_0}$.

\smallskip
Let $K_k$ be the complete graph on vertex set $\set{V_1,\dots, V_k}$. Consider the following $(q_1+2)$-coloring of the edges of $K_k$ with the color palette $\set{c_1,\dots, c_{q_1+2}}$. If the pair $(V_i, V_j)$ is $(\eps, p_e)$-regular in all $G_s$ for $s\in \CycleColors$ and has $p_e$-density at least $\rho$ in some $G_s$, give the edge between $V_i$ and $V_j$ in $K_k$ color $c_s$ (breaking ties arbitrarily). If the pair $(V_i,V_j)$ is $(\eps, p_e)$-regular in  $G_s$ for all $s\in \CycleColors$, but its $p_e$-density is less than $\rho$ in every such $G_s$, then color the edge between $V_i$ and $V_j$ in $K_k$ with color $c_{q_1+1}$.
Finally, if $(V_i,V_j)$ is not $(\eps,p_e)$-regular in $G_s$ for some $s\in\CycleColors$, let the edge between $V_i$ and $V_j$ in $K_k$ have color $c_{q_1+2}$.

By the fact that $k\geq k_0$ and our choice of $k_0$ (from~\thref{lem:Ramsey}), we know that at least one of the following must occur:
\begin{enumerate}[label=(\alph*)]
    \item For some $s\in[q_1]$, there are at least $c k^\ell$ copies of $C_\ell$ in color $c_s$.\label{item:ProofGammaRamsey_option_cycles}
    \item There are at least $c k^h$ copies of $K_h$ that are monochromatic in color $c_{q_1+1}$.\label{item:ProofGammaRamsey_option_Kh}
    \item There are at least $c k^2$ edges of color $c_{q_1+2}$.\label{item:ProofGammaRamsey_option_K2}
\end{enumerate}

If~\ref*{item:ProofGammaRamsey_option_cycles} occurs for some color $c_s\in[q_1]$, the fact that $ck^\ell \geq ck_0^\ell > 0$, together with property~\ref{lem:propertiesGamma:contains_cycle} in~\thref{lem:propertiesGamma}, implies that there is a copy of $C_\ell$ in $\Gamma'$ in color $s$. It remains to show that neither of the other cases can occur.

First consider option~\ref*{item:ProofGammaRamsey_option_K2}. We know that there are at most $\eps\binom{k}{2}$ pairs $(V_i,V_j)$ that are not $(\eps, p_e)$-regular in  $G_s$ for some $s\in\CycleColors$, and we have
\[\eps\binom{k}{2} \leq \frac{1}{10}\rho \binom{k}{2} \leq \frac{1}{10}c \binom{k}{2} < ck^2,\]
where the first two inequalities follow by the definitions of $\eps$ and $\rho$. Hence, option~\ref*{item:ProofGammaRamsey_option_K2} is indeed impossible.

We now prove that option~\ref*{item:ProofGammaRamsey_option_Kh} cannot occur. Suppose it does. We estimate the number of edges of $\Gamma'$ corresponding to pairs of color $c_{q_1+1}$ in two different ways. First note that if there is an edge of color $c_{q_1+1}$ between vertices $V_i$ and $V_j$, then the $(\eps,p_e)$-regular pair $(V_i,V_j)$ has $p_e$-density at most $\rho$ in $G_s$ for each $s\in\CycleColors$. Hence, in total, the pair  $(V_i,V_j)$ has $p_e$-density at most $q_1\rho$ in $\Gamma'$. Hence, the number of edges in $\Gamma'$ between pairs $(V_i,V_j)$ corresponding to color $c_{q_1+1}$ is at most 
\begin{equation}
    \binom{k}{2}q_1\rho p_e m^2 = \binom{k}{2}q_1 \rho p_e \parens*{\frac{n}{k}}^2 < \frac{1}{2}q_1\rho A n^{1+1/(\ell-1)} = \frac{c}{4} A n^{1+1/(\ell-1)}.\label{eq:ub}
\end{equation}

Now, since option~\ref*{item:ProofGammaRamsey_option_Kh} occurs, we have at least $ck^h$ copies of $K_h$ that are monochromatic in color $c_{q_1+1}$ in $K_k$. Denote these by $K_h^1,K_h^2,\ldots,K_h^x$, where $x=\ceil{ck^h}$. The vertex set $V(K_h^i)$ of each such copy gives an $h$-partite subgraph $J_i\subseteq \Gamma$ induced by the sets $V_j$ corresponding to the vertices of $K_h^i$.
As each partite set of $J_i$ has size $m\geq \frac{n}{K_0}\geq\frac{n}{\log n}$,~\thref{lem:propertiesGamma}~\ref{lem:propertiesGamma:copies_of_H} guarantees that $J_i$ contains a family $\mathcal{H}_i$ of at least $\frac{1}{4}m^hp_h$ distinct hyperedge copies of $H$, for every $i\in [x]$.
As each hyperedge copy in $\mathcal{H}_i$ intersects each partite set of $K_h^i$, it is immediate that $\mathcal{H}_i\cap\mathcal{H}_j\neq\emptyset$ for $i\neq j$. Hence, there exist $\size*{\bigcup_{i\in[x]}\mathcal{H}_i}\geq\frac{1}{4}ck^hm^hp_h$ copies of $K_h$ in $\Gamma$. Since every copy of $K_h$ in $\Gamma$ is a hyperedge copy and no two hyperedge copies share an edge, we find that $\Gamma'$ has at least  
\begin{align}
     \frac{1}{4}ck^hm^hp_h \geq c k^h \frac{1}{4} \parens*{\frac{n}{k}}^h A n^{-h+1 +1/(\ell-1)} = \frac{c}{4}A n^{1+1/(\ell-1)}
\end{align}
edges corresponding to pairs $(V_i,V_j)$ in color $c_{q_1+1}$, contradicting \eqref{eq:ub}.\end{proof}

\subsection{Construction of set-determiners}
This section uses ideas from~\cite{siggers_non-bipartite_2014} to prove~\thref{lem:cycle_and_clique_determiners}.  Recall that \CycleColors{} and \CliqueColors{} denote the cycle-colors $\{1,\ldots,q_1\}$ and clique-colors $\{q_1+1,\ldots,q\}$, respectively. 
By construction and by~\thref{lem:propertiesGamma}, we know that $\Gamma$ satisfies the following properties:
\begin{enumerate}[label=(\roman*)]
    \item Every copy of $K_t$ in $\Gamma$ is a hyperedge copy.\label{item:Construction_det_NoH_NoKt}
    \item Every copy of $C_{\ell'}$ for $\ell'<\ell$ is a hyperedge copy.\label{item:Construction_det_NoCycle}
    \item Each edge of $\Gamma$ belongs to a unique copy of $K_h$.\label{item:Construction_det_uniqueH}
\end{enumerate}

Now, let $G\subseteq \Gamma$ be a minimal $q$-Ramsey graph for the $q$-tuple $\cT(q_1,q_2,\ell, t)$; it is not difficult to see that $G$ satisfies properties~\ref*{item:Construction_det_NoH_NoKt} and~\ref*{item:Construction_det_NoCycle} given above. In fact, we have a good understanding of what $G$ needs to look like, as given in the following lemma. Naturally, the lemma also establishes that $G$ satisfies property~\ref*{item:Construction_det_uniqueH} above. 

\begin{lemma}\thlabel{lem:structureG} 
The graph $G$ is the union of hyperedge copies of $K_h$, that is, every edge of $G$ belongs to a hyperedge copy of $K_h$ in $G$.
\end{lemma}
\begin{proof}
Suppose there is an edge $e$ that does not belong to a copy of $K_h$ in $G$. We know that $e$ \emph{does} belong to a copy of $K_h$ in $\Gamma\supseteq G$; let $H$ denote this copy of $K_h$ in $\Gamma$ and let $F$ denote the set of edges on $V(H)$ that are in $\Gamma$ but not in $G$. Notice that $\emptyset \subsetneq F \subsetneq E(H)$ by our assumption. 

By the minimality of $G$, we know that $G-H$ has a $\cT$-free $q$-coloring $\phi$. Additionally, since $K_h$ is minimal $q_2$-Ramsey for $K_t$, the graph $H-F$ has a $K_t$-free $q_2$-coloring $\phi': E(H-F) \rightarrow \CliqueColors{}$. We now define a $q$-coloring $\widetilde{\phi}$ of $G$ by setting $\widetilde{\phi} = \phi\cup\phi'$.

We claim that $\widetilde{\phi}$ is a $\cT$-free $q$-coloring of $G$.
Indeed, since $\phi$ is a $\cT$-free coloring of $G-H$, there are no monochromatic cycles in any cycle-color, and since in the coloring of $H-F$ we add no further edges in these colors, we know that there are no monochromatic copies of $C_\ell$ in any cycle-color in all of $G$. Furthermore, since there are no non-hyperedge copies of $K_t$ in $G$ and neither $\phi$ nor $\phi'$ contains a monochromatic copy of $K_t$ in any color in $\CliqueColors$, we know that there are also no monochromatic copies of $K_t$  in any clique-color in all of $G$. Hence $\widetilde{\phi}$ is a $\cT$-free $q$-coloring of $G$, contradicting the fact that $G\rightarrow_q\cT$.
\end{proof}

Now, let $e$ be a fixed edge of $G$ and let $H$ be the copy  of $K_h$ in $G$ containing $e$. Let $D$ be the graph obtained from $G$ by removing all edges of $H$ except for $e$, that is, $D = G - (H-e)$.  We now claim that $D$ is a \CliqueColors-determiner for the tuple $\cT$. This construction generalizes the one presented by Siggers~\cite{siggers_non-bipartite_2014}.

\begin{lemma}\thlabel{lem:clique-determiner}
The graph $D$ is a safe \CliqueColors-determiner for the tuple $\cT$ with signal edge $e$.
\end{lemma}
\begin{proof}
We first show property~\ref{axiom:set_determiner_colored_edge}. For a contradiction, suppose $\psi$ is a $\cT$-free coloring of $D$ in which $\psi(e) \in \CycleColors{}$. Then, by an argument similar to the one used in~\thref{lem:structureG}, putting together this $\cT$-free coloring of $D$ and a $K_t$-free $q_2$-coloring of $H-e$ (with colors in \CliqueColors{}), we obtain a $\cT$-free coloring of $G$, which is a contradiction to the fact that $G\rightarrow_q \cT$.

To see properties~\ref{axiom:set_determiner_not_Ramsey} and~\ref{axiom:set_determiner_any_color}, note that $D$ is a proper subgraph of $G$, so  $D$ has a $\cT$-free $q$-coloring $\phi$. Further, by permuting the clique-colors in $\phi$ appropriately, we can obtain a $\cT$-free coloring of $D$ in which the edge $e$ has any color in $\CliqueColors{}$. 

It remains to show that $\phi$ is safe at $\set*{e}$. Let $F$ be any graph such that $V(D)\cap V(F)\subseteq e$. Let $\phi'$ be a $\cT$-free $q$-coloring of $F$ that agrees with $\phi$ on the edge $e$. We claim that the coloring $\widetilde{\phi}$, given by $\widetilde{\phi} = \phi\cup\phi'$,
is a $\cT$-free $q$-coloring of $D\cup F$. We know that the restrictions of $\widetilde{\phi}$ to both $D$ and $F$ are $\cT$-free; it remains to show that there are no monochromatic cliques or cycles in the appropriate colors intersecting both $V(D)-e$ and $V(F)-e$. 

First, it is not difficult to see that there can be no such copy of $K_t$. For $t=3$, this is clear. If $t\geq 4$ and there is a $t$-clique $K$ intersecting both $D-e$ and $F-e$, then we can disconnect $K$ by removing the vertices of $e$, which is impossible. Suppose there is such a copy $C$ of $C_\ell$. Note first that $C$ must contain both vertices of $e$ because $C_\ell$ is $2$-connected. Now, let $v$ be a vertex of $C$ contained in $V(D)-e$, and let $w$ be a vertex of $C$ contained in $V(F)-e$. Now, there are no non-hyperedge cycles of length less than $\ell$ in $D$, so every cycle containing $e$ in $D$ has length at least $\ell$. Hence, the vertices $v$ and $w$ cannot be contained in a cycle of length $\ell$ with both endpoints of $e$, and therefore $C$ cannot exist. 
Thus the coloring $\widetilde{\phi}$ is $\cT$-free, implying that $\phi$ is safe. This completes the verification of the safeness property.
\end{proof}

Now we construct a safe \CycleColors-determiner $D'$ by taking a copy $H$ of $K_h$, fixing one edge $f$, and attaching copies of the \CliqueColors-determiner $D$ constructed above to all remaining edges of $H$. This again generalizes a construction of Siggers~\cite{siggers_non-bipartite_2014}.

\begin{lemma}\thlabel{lem:cycle-determiner}
The graph $D'$ is a safe \CycleColors{}-determiner for the tuple $\cT$ with signal edge $f$.
\end{lemma}
\begin{proof}
We again begin with property~\ref{axiom:set_determiner_colored_edge}. 
Take an arbitrary $\cT$-free coloring of $D'$. This coloring induces a $\cT$-free coloring on each copy of $D$, so, by property~\ref{axiom:set_determiner_colored_edge} of $D$, all edges of $H-f$ have colors in $\CliqueColors{}$. If $f$ has one of these colors too, then $H$ is fully colored with colors in $\CliqueColors{}$. Since $H$ is $q_2$-Ramsey for $K_t$, there exists a monochromatic copy of $K_t$ in $H$, contradicting the fact that the coloring $\phi$ is $\cT$-free. So the color of $f$ must be in the set $\CycleColors{}$.

We show properties~\ref{axiom:set_determiner_not_Ramsey} and~\ref{axiom:set_determiner_any_color} next. By minimality, we know that $H-f$ is not $q_2$-Ramsey for $K_t$, and hence it has a $K_t$-free coloring $\psi$ from the palette~\CliqueColors. Let $\phi$ be a $q$-coloring extending $\psi$ in which each copy of the determiner $D$ has a safe $\cT$-free coloring and the edge $f$ has an arbitrary color from $\CycleColors{}$; this coloring $\widetilde{\phi}$ exists by property~\ref{axiom:set_determiner_any_color} of $D$. Since the coloring of each copy of $D$ is safe and since $H$ has a $\cT$-free $q$-coloring, the coloring $\phi$ of $D'$ is also $\cT$-free.

Finally, to see the safeness of $\phi$, let $F$ be a graph such that $V(D')\cap V(F)\subseteq f$. If $F$ is given a $\cT$-free $q$-coloring $\phi'$ that agrees with $\phi$ on $f$, then the coloring $\widetilde{\phi} = \phi\cup \phi'$ is a $\cT$-free $q$-coloring of $D' \cup F$. Indeed, since each copy of $D$ is safe and the only edge of $H$ that has color in $\CycleColors{}$ is $f$, we know that there can be no monochromatic copy of $C_\ell$ in $D'\cup F$ using a cycle-color in $\widetilde{\phi}$. Similarly, since we cannot disconnect $K_t$ by removing at most two vertices, we know that there can be no copy of $K_t$ intersecting both $V(D')-f$ and $V(F)-f$, and hence there can be no monochromatic copy of $K_t$ in a clique-color in $\widetilde{\phi}$. Hence, $\widetilde{\phi}$ is a $\cT$-free $q$-coloring and thus $\phi$ is a safe coloring of $D'$.
\end{proof}

\subsection{Construction of set-senders}

So far we have constructed a \CliqueColors{}-determiner $D$ and a \CycleColors-determiner $D'$, generalizing ideas from~\cite{bollobas2001ramsey} and~\cite{siggers_non-bipartite_2014}. We now take the constructions a step further and use our set-determiners to build set-senders for these sets of colors when $q_1>1$ or $q_2>1$, proving~\thref{lem:cycle_and_clique_senders}. 

If $q_1>1$, let $S$ be a safe negative (respectively positive) signal sender for $C_\ell$ with $q_1$ colors, as guaranteed by~\thref{lem:signal_senders_cycles} and Remark~\ref{rem:safeness}; let $e$ and $f$ denote its signal edges. Let $R$ be a graph obtained from $S$ by attaching a copy of $D'$ to every edge of $S$.

\begin{lemma}\thlabel{lemma:Existence_cycle_senders_for_tau}
If $S$ is a negative (respectively positive) signal sender for $C_\ell$ with signal edges $e$ and $f$ as above, then $R$ is a safe negative (respectively positive) \CycleColors-sender for $\cT$ with signal edges $e$ and $f$. 
\end{lemma}
\begin{proof}

Assume $S$ is a negative signal sender for $C_\ell$ in $q_1$ colors; the other case is similar. 
We first show properties~\ref{axiom:set_sender_not_Ramsey} and~\ref{axiom:set_sender_any_color}. Let $c_1, c_2\in \CycleColors$ be distinct. We know that $S\nrightarrow_{q_1} C_\ell$, so $S$ has a safe $C_\ell$-free coloring from the set $\CycleColors{}$, and by property~\ref{axiom:set_sender_any_color} of $S$, we can ensure that $e$ and $f$ receive colors $c_1$ and $c_2$, respectively. Now, since the signal edge of each copy of $D'$ has color in $\CycleColors{}$, by property~\ref{axiom:set_determiner_any_color} of $D'$, this coloring of $S$ can be extended to each copy of $D'$ so that each copy of $D'$ has a safe $\cT$-free $q$-coloring. The coloring of each copy of $D'$ is safe, so the $q$-coloring defined on $R$ is $\cT$-free. To see the safeness of this coloring, notice that the coloring of each copy $D'$ is safe at its signal edge and the coloring of $S$, containing only colors from $\CycleColors$, is safe at $\set{e,f}$. \end{proof}

Finally, if $q_2>1$, we build \CliqueColors-senders for $\cT$. Let $S'$ be a safe negative (respectively positive) signal sender for $K_t$ with $q_2$ colors taken as \CliqueColors, as guaranteed by~\thref{lem:signal_senders_cliques} and Remark~\ref{rem:safeness}; let $e$ and $f$ denote its signal edges. Let $R'$ be a graph obtained from $S'$ by attaching a copy of $D$ to every edge of $S'$. We omit the proof that $R'$ is a set-sender for $K_t$, as it is essentially the same as that of~\thref{lemma:Existence_cycle_senders_for_tau}.

\begin{lemma}\thlabel{lemma:Existence_clique_senders_for_tau}
If $S'$ is a negative (respectively positive) signal sender for $K_t$ with signal edges $e$ and $f$, then $R'$ is a safe negative (respectively positive) \CliqueColors-sender $R'$ for $\cT$ with signal edges $e$ and $f$.
\end{lemma}

\section{Concluding remarks}

In this paper, we initiated the study of the parameter $s_q$ in the asymmetric setting for tuples consisting of cliques and cycles. The upper and lower bounds we obtain are strongly dependent on the existing bounds for the symmetric parameter $s_q(K_t)$. As noted by the authors in \cite{fox2016minimum}, the study of $s_q(K_t)$ appears to be tightly connected to the \Erdos{}-\Rogers{} function, implying that any improvements on our current results would probably be non-trivial. We refer to \cite[Section 5]{fox2016minimum} for a more detailed discussion on the relationship between $s_q(K_t)$ and the \Erdos{}-\Rogers{} function.\smallskip

It would be desirable to study other asymmetric cases of the problem, and a natural place to start is to consider pairs of graphs for which safe determiners are known to exist (including all pairs of $3$-connected graphs and the pairs considered by Siggers in~\cite{siggers_non-bipartite_2014}). \smallskip

The multicolor asymmetric setting offers even more room for study, as the existence of gadget graphs is an open problem even in some very natural cases. Our method allows us to construct set-determiners and set-senders for tuples of the form $(C_\ell,\dots, C_\ell, K_s, K_t)$. However we are not aware of a way to build gadget graphs for asymmetric $q$-tuples of cliques, with $q>2$. Since studying Ramsey graphs for cliques is a central theme in Ramsey theory, we believe that resolving the following problem would be of interest. 

\begin{problem}
Construct signal senders for asymmetric $q$-tuples $(K_{t_1},\dots, K_{t_q})$.
\end{problem}
The natural first instances to attack, which might also shed some light on the general case, are tuples of the form $(K_t, \ldots,K_t, K_k)$ or $(K_t, K_s, K_k)$. Once we have the necessary tools, it would be very interesting to investigate the parameter $s_q$ for such tuples.

It would also be desirable to determine if the upper bound in~\thref{thm:sq_relation} holds in other cases. In particular, it was conjectured by~\FGLPS~\cite{fox2016minimum} that $s_q(K_{t-1}) \leq s_q(K_t)$ for $q>3$. Perhaps the following asymmetric version would be more approachable.

\begin{problem}
Show that $s_q(\underbrace{K_{t-1},\ldots,K_{t-1}}_{q_1+1\text{ times}},\underbrace{K_t,\ldots,K_t}_{q_2-1\text{ times}}) \leq s_q(\underbrace{K_{t-1},\ldots,K_{t-1}}_{q_1\text{ times}},\underbrace{K_t,\ldots,K_t}_{q_2\text{ times}})$.
\end{problem}

\section*{Acknowledgements} The second author was supported by the Deutsche Forschungsgemeinschaft (DFG) Graduiertenkolleg ``Facets of Complexity'' (GRK 2434). The fifth author was supported by the Commonwealth through an Australian Government Research Training Program Scholarship. The sixth author was partially supported by the Australian Research Council.


\bibliographystyle{amsplain}
\bibliography{Bibliography.bib}

\appendix
\section{Proof of Lemma~\ref{lem:propertiesGamma}~\ref{lem:propertiesGamma:contains_cycle}}\label{app:proof_lemma}

We now give the proof of~\thref{lem:propertiesGamma}~\ref{lem:propertiesGamma:contains_cycle}. The proof is similar to the proof of Proposition 9 in~\cite{bollobas2001ramsey}, but we use modern results related to the K\L{}R conjecture.

\begin{proof}[Proof of~\thref{lem:propertiesGamma}~\ref{lem:propertiesGamma:contains_cycle}.]

    Let $m$ satisfy $\frac{n}{K_0}\leq m\leq \frac{n}{k_0}$; we can write $p_e = Bm^{-1+1/(\ell-1)}$, where $B = A\parens*{\frac{n}{m}}^{-1+1/(\ell-1)}$. Notice that $B$ satisfies $AK_0^{-1+1/(\ell-1)} \leq B \leq Ak_0^{-1+1/(\ell-1)}$.

    Let $(V_1,\dots, V_\ell)$ and $\Gamma'$ be as given.  Suppose that the pairs $(V_i,V_{i+1})$ for $i\in [\ell]$ are $(\eps, p_e)$-regular with $p_e$-density at least $\rho$ in $\Gamma'$. Then we have $e_{\Gamma'}(V_i,V_{i+1}) \geq \rho p_e m^2$ for all $i\in [\ell]$. Let $M$ be an integer satisfying
    \begin{align*}
        \rho p_e m^2 \leq M \leq \min\limits_{i\in [\ell]}e_{\Gamma'}(V_i,V_{i+1}).
    \end{align*}
    Notice that this integer $M$ satisfies
    \begin{align*}
        M &\geq \rho p_e m^2  = \rho B m^{1+1/(\ell-1)}  
        \geq \rho A K_0^{-1+1/(\ell-1)} m^{1+1/(\ell-1)}\\ &\geq Cm^{1+1/(\ell-1)}\geq 2 \beta m = \beta \size*{V_i\cup V_{i+1}},
    \end{align*}
    since $A \geq K_0^{1-1/(\ell-1)}C/\rho$ and $n$, and hence $m$, is taken to be sufficiently large. 
    
    Consider the pair $(V_1,V_2)$ and let $d = \frac{e_{\Gamma'}(V_1,V_2)}{m^2}$; then we have $d \geq \rho p_e$, and thus $p_e\leq \frac{d}{\rho}$. By definition, it then follows that the pair $(V_1,V_2)$ is $\parens*{\frac{\eps}{\rho}, d}$-regular, or simply $\parens*{\frac{\eps}{\rho}}$-regular. 
    By~\thref{lem:feweredges}, there is a subset $E_{1,2}\subseteq E_{\Gamma'}(V_1,V_2)$ such that $|E_{1,2}| = M$ and the pair $(V_1,V_2)$ is $\parens*{\frac{2\eps}{\rho}}$-regular in $(V_1\cup V_2, E_{1,2})$. Repeating this argument for all pairs of the form $(V_i,V_{i+1})$, we find that $\Gamma'[V_1\cup\dots\cup V_\ell]$ contains at least one graph in $\mathcal{G}\parens*{\ell,m, (V_i)_{i=1}^\ell,M,\frac{2\eps}{\rho}}$.
    
    Our goal now is to show that, with high probability, there is no collection of subsets $(V_i)_{i=1}^\ell$ and subgraph $\Gamma'\subseteq \Gamma$ as given in the statement such that $\Gamma'[V_1\cup\dots\cup V_\ell]$ contains a subgraph belonging to $\mathcal{F}\parens*{\ell,m, (V_i)_{i=1}^\ell,M,\frac{2\eps}{\rho}}$. Again, let the $\ell$-tuple $(V_1,\dots, V_\ell)$ be fixed.
    If $F\in \mathcal{F}\parens*{\ell,m, (V_i)_{i=1}^\ell,M,\frac{2\eps}{\rho}}$ has edges $e_1,\dots, e_{M\ell}$ and there exists a transversal $\Gamma'$ such that $F\subseteq \Gamma'[V_1\cup\dots\cup V_\ell]$, there must exist distinct hyperedges $\mathcal{E}_1,\dots, \mathcal{E}_{M\ell}\in E(\mathcal{H}_{n,p_h})$ such that $e_i\subseteq \mathcal{E}_i \text{ for all }i\in[M\ell]$. Therefore
    \begin{align}
        \Pr\brackets*{\exists \text{transversal } \Gamma' :\ F\subseteq \Gamma'[V_1\cup\dots\cup V_\ell]}
        &\leq \parens*{\binom{n-2}{h-2}p_h}^{M\ell} \nonumber \\
        &\leq \parens*{(n-2)^{h-2}An^{-(h-1)+1/(\ell-1)}}^{M\ell} \nonumber \\
        &\leq \parens*{An^{-1+1/(\ell-1)}}^{M\ell} = p_e^{M\ell}. \label{eq:prob_exist_gammae_with_subgraph}
    \end{align}
    Note that, when $n$ is sufficiently large, we have $m\geq m_0$. By choice of $\eps\leq\rho\eps_0/2$, applying~\thref{lem:counting} and the union bound, we obtain
    \begin{align*}
        &\Pr\brackets*{\exists\ \text{transversal }\Gamma', F\in\mathcal{F}\parens*{\ell,m, (V_i)_{i=1}^\ell,M,\frac{2\eps}{\rho}}: F\subseteq \Gamma'[V_1\cup\dots\cup V_\ell]} \\
        &\leq \alpha^M\binom{m^2}{M}^\ell p_e^{\ell M} \leq \alpha^M \parens*{\frac{m^2e}{M}}^{\ell M}p_e^{\ell M}
        \leq \alpha^M \parens*{\frac{e}{\rho}}^{\ell M} = e^{-M},
    \end{align*}
    where the last inequality follows from the fact that $M\geq \rho p_e m^2$ and the final step follows by the choice of $\alpha$.

    This implies that, for any fixed integers $m$ and $M$ and any collection of disjoint subsets $V_1,\dots, V_\ell$ of $[n]$, each of size $m$, the probability that there exists a transversal $\Gamma'$ such that $\Gamma'[V_1\cup\dots\cup V_\ell]$ contains some graph in $\mathcal{F}\parens*{\ell,m, (V_i)_{i=1}^\ell,M,\frac{2\eps}{\rho}}$ is at most $e^{-M}$. 
    
    Now, for any choice of $\frac{n}{K_0} \leq m\leq \frac{n}{k_0}$ and $Cm^{1+1/(\ell-1)}\leq M\leq m^2\leq n^2$, there are at most $n^{m\ell}$ choices for the sets $V_1,\dots, V_\ell$. Summing over the possible choices for the sets $V_1,\dots, V_\ell$ and the possible choices for $m$ and $M$, we find that the probability that~\ref{lem:propertiesGamma:contains_cycle} fails is bounded from above by the probability that there exist $m$, $M$, $(V_i)_{i=1}^\ell$ and $\Gamma'$ such that $\Gamma'[V_1\cup\dots\cup V_\ell]$ contains a member of $\mathcal{F}\parens*{\ell,m, (V_i)_{i=1}^\ell,M,\frac{2\eps}{\rho}}$, which is at most
    \begin{align*}
        \sum\limits_{m} \sum\limits_{M} n^{m \ell}e^{-M} 
        &\leq  \sum\limits_{m} \sum\limits_{M} \exp(- Cm^{1+1/(\ell-1)} + m\ell \log n) \\
        &\leq \sum\limits_{m} \sum\limits_{M} \exp\parens*{-C\parens*{\frac{n}{K_0}}^{1+1/(\ell-1)} + \frac{n}{k_0}\ell \log n}\\ 
        &\leq n^3 \exp\parens*{-C\parens*{\frac{n}{K_0}}^{1+1/(\ell-1)} + \frac{n}{k_0}\ell \log n} = o(1).
    \end{align*}
    
\end{proof}

\end{document}